\documentclass[a4paper,11pt,reqno]{amsart}   

\usepackage[latin1]{inputenc}
\usepackage[T1]{fontenc}

\usepackage{amsmath}
\usepackage{amssymb}
\usepackage{amsthm}
\usepackage{mathabx}
\usepackage{MnSymbol}

\usepackage{color}
\usepackage{tikz}

\theoremstyle{plain}
\newtheorem{theorem}{Theorem}[section]
\newtheorem{lemma}[theorem]{Lemma}
\newtheorem{proposition}[theorem]{Proposition}
\newtheorem{corollary}[theorem]{Corollary}

\theoremstyle{definition}

\newtheorem{definition}[theorem]{Definition}
\newtheorem{example}[theorem]{Example}

\theoremstyle{remark}
\newtheorem{remark}[theorem]{Remark}

\newcommand{\ssyt}{semi--standard Young tableau}

\author[L. Colmenarejo]{Laura Colmenarejo} 
\thanks{Partially supported by MTM2013-40455-P, P12-FQM-2696, FQM-333, and FEDER. }

\title[Combinatorics on Kronecker coefficients]{\textsc{Combinatorics on several families of \\
Kronecker coefficients \\
related to plane partitions}}
\email{laurach@us.es}
\address{Department of Algebra, University of Seville}

\begin{document}

\begin{abstract}
{\it We present a study of three families of Kronecker coefficients, which we describe in terms of reduced Kronecker coefficients. This study is grounded on the generating function of the coefficients, proved by a bijection between two combinatorial objects. This study includes the connection between plane partitions and these three families of reduced Kronecker coefficients, providing us their combinatorial interpretation. As an application, we verify that the saturation hypothesis holds for our three families of reduced Kronecker coefficients. The study also includes other interpretation in terms of the quasipolynomials that govern these families. We specify the degree and the period of these quasipolynomials. Finally, the direct relation between Kronecker coefficients and reduced Kronecker coefficients allows us to give some observations about the rate of growth of the Kronecker coefficients associated to the reduced Kronecker coefficients of the study. }

\end{abstract}

\maketitle

\section*{Introduction}

The Kronecker product of two Schur functions $s_\mu$ and $s_\nu$, denoted by $s_\mu \ast s_\nu$, is the Frobenius characteristic of the tensor product of the irreducible representations of the symmetric group corresponding to the partitions $\mu$ and $\nu$. The coefficient of $s_\lambda$ in this product, $g_{\mu\nu}^\lambda$, is called \emph{Kronecker coefficient}, and corresponds to the multiplicity of the irreducible character $\chi^\lambda$ in $\chi^\mu \chi^\nu$. 

Although the problem has been studied for almost a century, a formula for decomposing the Kronecker product is unavailable. Some partial results can be found in  \cite{IniLiu,MR2264933,Blasiak12,zbMATH06062487,MR2721516,MR582085,MR977863,MR1315363,MR1798227}. Recently, the problem of understanding the Kronecker coefficients has come back to the forefront also because of their connections to Geometric Complexity Theory, \cite{zbMATH06062487,MR2721467,PakPanovaComple14,MR2570451}, and to Quantum Information Theory, \cite{MR2197548,Klya04,Wigner}.

 In 1938, Murnaghan discovered an intriguing stabilization phenomena for the Kronecker coefficients,  \cite{MR1507347, MR0075213}. They stabilize when we increase the first rows of its three indexing partitions. The limits of these sequences are known as the \emph{reduced Kronecker coefficients}.
 They are indexed by the partitions obtained after deleting the first parts of the original triple.  

In the present work, we study three families of Kronecker coefficients (including $\overline{g}_{(k^a),(k^b)}^{(k)}$ that was presented in \cite{CR2015}). These families of Kronecker coefficients can be described in terms of reduced Kronecker coefficients. These families exhibit stabilization phenomena that we show later. We study the following families of reduced Kronecker coefficients:
\begin{itemize}
\item[$\triangleright$] The {\bf Family 1} is formed by the reduced Kronecker coefficients of the form $\overline{g}_{(k^a),(k^b)}^{(k)}$. This family includes all families of Kronecker coefficients of the form $g_{(N-k\cdot a, k^a),(N-k\cdot b, k^b)}^{(N-k,k)}$, for $N$ large enough.  \vspace{0.3cm}
\item[$\triangleright$] The {\bf Family 2} is formed by the reduced Kronecker coefficients of the form $\overline{g}_{((k+i)^a), (k^b)}^{(k)}$. This family includes all families of Kronecker coefficients of the form $g_{(N-(k+i)\cdot a, (k+i)^a),(N-k\cdot b, k^b)}^{(N-k,k)}$, for $N$ large enough. \vspace{0.3cm}
\item[$\triangleright$] The {\bf Family 3} is formed by the reduced Kronecker coefficients of the form $\overline{g}_{(k^b), (k+i,k^{a})}^{(k)}$. This family includes all families of Kronecker coefficients of the form $g_{(N-k\cdot b, k^b),(N-(k+1)\cdot a-i, k+i, k^a)}^{(N-k,k)}$, for $N$ large enough.
\end{itemize} 

A general bound for $N$ in given in Corollary \ref{StabilityBound} and specific ones for each family are given in the proofs of Theorem \ref{ThmGF} and Theorem \ref{ThmFam3}.

  After a briefly introduction about the reduced Kronecker coefficients and their relation with the Kronecker coefficients, Section 2 focuses on our families of reduced Kronecker coefficients. Our approach uses as main combinatorial tool the \emph{Kronecker tableaux}, Definition \ref{KT}, of C. Ballantine and R. Orellana. We find this combinatorial approach through special kind of tableaux in a paper of C. Ballantine and B. Hallahan, \cite{BallHalla12}, where they use \emph{Yamanouchi coloured tableaux} and Blasiak's combinatorial rule, \cite{Blasiak12}, to study the stability of the Kronecker product of a Schur function indexed by a hook partition and another Schur function indexed by a rectangle partition. In our case, we use the description of a special type of Kronecker coefficients in terms of Kronecker tableaux presented in Theorem \ref{ThmKT} to compute the reduced Kronecker coefficients that we consider. 

  The generating function related to Families 1, 2 and 3 are shown in Theorem \ref{ThmGF} and Theorem \ref{ThmFam3}. For Family 1, the generating function depends on $a$ and $b$. In fact, once we fix $k$, the reduced Kronecker coefficients for $b=a$ in Family 1 are stable for $a$ large enough. In case of Family 2 for $b=a$, once we fix $i$, there are some initial zeros and the rest of the sequence has as generating function the same generating function than Family 1 for $b=a$. The number of initial zeros depends on $i$ and $a$. Then, results related to Family 1 apply also for Family 2, once we shift these initial zeros. In Table \ref{table3} it is presented the stability phenomenon of Family 3: the reduced Kronecker coefficients stabilize in the diagonals. We give the generating function of the stable values of the diagonals.  

In Section \ref{PP} we present a striking connection between our families of reduced Kronecker coefficients and plane partitions. This unexpected relation is showed in Theorem \ref{ThmPP}, which is proved comparing the generating functions obtained in Theorem \ref{ThmGF} and in Theorem \ref{ThmFam3} with a reformulation of MacMahon's classical formula for the generating function of plane partitions, Corollary \ref{ThmPPfg}. Plane partitions have appeared before in the study of the Kronecker coefficients in works of E. Vallejo, \cite{MR1747064}, and L. Manivel, \cite{MR2550164}. As an application, we use the combinatorial interpretations in terms of plane partitions to prove that the saturation hypothesis of Kirillov and Klyachko holds for these three families of reduced Kronecker coefficients, and  that the resulting sequences  are weakly increasing in case of $b=a$ in Family 1 and $b=a+1$ in Family 3. Note that the saturation hypothesis do not hold for Kronecker coefficients in general, \cite{MR2570451}.

In Section 4, we present other interpretation of our reduced Kronecker coefficients in terms of the quasipolynomials, Theorem \ref{ThmQuasiPoly}. We specify the degree and the period of these quasipolynomials.

Finally, moving back to the setting of Kronecker coefficients, in Section 5, we translate our main results in terms of Kronecker coefficients and we present some implications towards the study of the rate of growth of the Kronecker coefficients of this study. 

The results concerning Family 1 are joint work with M. Rosas, and were announced in \cite{CR2015}. Complete details of the proofs can be found in the thesis of L. Colmenarejo, \cite{Col16}.


\section{Reduced Kronecker coefficients}\label{RKC}

The reduced Kronecker coefficients, $\overline{g}_{\alpha \beta}^\gamma$, were introduced by Murnaghan through the Kronecker product.
\begin{theorem}[Murnaghan's Theorem, \cite{MR1507347} \cite{MR0075213}]\label{MurnThm}
There exists a family of non--negative integers $\{ \overline{g}_{\alpha \beta}^\gamma\}$, indexed by triples of partitions $(\alpha, \beta, \gamma)$, such that, for $\alpha$ and $\beta$ fixed, only finitely many terms 
$\overline{g}_{\alpha \beta}^\gamma$ are non--zero, and for all $n\geq 0$,
\begin{eqnarray*}
 s_{\alpha[n]} \ast s_{\beta[n]} = \sum_{\gamma} \overline{g}_{\alpha \beta}^\gamma s_{\gamma[n]},
\end{eqnarray*}
where $\alpha[n]= \left(n-|\alpha|, \alpha_1,\alpha_2,\dots\right)$. 
\end{theorem}
Note that $\alpha[n]$ is a partition if and only if $n\geq \alpha_1 + |\alpha|$. 

Murnaghan's theorem shows the following stability property for the Kronecker coefficients: for $n$ big enough the expansion of $s_{\alpha[n]} \ast s_{\beta[n]}$ in the Schur basis does not depend on the first part of the indexing partitions. 

In particular, given three partitions $\alpha$, $\beta$ and $\gamma$, the sequence $\left\{ g_{\alpha[n] \beta[n]}^{\gamma[n]}\right\}_n$ is eventually constant. The reduced Kronecker coefficient $\overline{g}_{\alpha \beta}^\gamma$ can be defined as the stable value of this sequence. 
Therefore, there exists a positive integer $N$ such that, for $n\geq N$,
\begin{eqnarray*}
 \overline{g}_{\alpha \beta}^\gamma = g_{\alpha[n] \beta[n]}^{\gamma[n]}.
\end{eqnarray*}

The point at which the expansion of the Kronecker product $s_{\alpha[n]} \ast s_{\beta[n]}$ stabilizes is denoted by $stab(\alpha, \beta)$. In \cite{MR2774644}, E. Briand, R. Orellana, and M. Rosas prove that $stab(\alpha,\beta)=|\alpha| + |\beta| + \alpha_1 + \beta_1$. Since the reduced Kronecker coefficients inherit the symmetry from the Kronecker coefficients, we have the following bound for $N$.
\begin{corollary}\label{StabilityBound}
 Consider three partitions $\alpha$, $\beta$ and $\gamma$ such that, for $n \geq N$
 \begin{eqnarray*}
 \overline{g}_{\alpha \beta}^\gamma = g_{\alpha[n] \beta[n]}^{\gamma[n]}.
\end{eqnarray*}
Then, $N \leq \min \{ stab(\alpha,\beta), stab(\alpha, \gamma), stab(\beta, \gamma) \}$. 
\end{corollary}

Remark that Kronecker coefficients and reduced Kronecker coefficients are directly related. The reduced Kronecker coefficients are interesting objects of their own right. For instance, Littlewood observed that they coincide with the Littlewood--Richardson coefficients when $|\alpha|+|\beta| = |\gamma|$, \cite{MR0070640, MR0095209}. Furthermore, they contain enough information to compute from them the Kronecker coefficients, \cite{MR2774644}.


\section{The generating function of several families \\ of reduced Kronecker coefficient}\label{GF}

Before computing the generating functions of these families, we need to introduce the combinatorial tools that we use in our proofs. 
\begin{definition}
An \emph{$\alpha$--lattice permutation} is a sequence of integers such that in every initial part of the sequence the number of occurrences of $i$ plus $\alpha_i$ is bigger or equal than the number of occurrences of $i+1$ plus $\alpha_{i+1}$.
\end{definition}

\begin{definition}\label{KT}
A \emph{Kronecker tableau} is a semi--standard Young tableau  $T$ of shape $\lambda/ \alpha$ and type $\nu/\alpha$, with $\alpha \subseteq \lambda \cap \nu$, whose reverse reading word is an $\alpha$--lattice permutation, and such that if either $\alpha_1=\alpha_2$, or $\alpha_1>\alpha_2$ and any one of the following two conditions is satisfied:
\begin{itemize}
\item[(i)]  The number of $1$'s in the second row of $\lambda/\alpha$ is exactly $\alpha_1-\alpha_2$.
\item[(ii)]  The number of $2$'s in the first row of $\lambda/\alpha$ is exactly $\alpha_1-\alpha_2$.
\end{itemize}

We denote by $k_{\alpha \nu}^\alpha$ the number of Kronecker tableaux of shape $\lambda/\alpha$ and type $\nu/\alpha$, with $\alpha \subseteq \lambda \cap \nu$.
\end{definition}

For instance, consider $\lambda = (5,3,2,1)$, $\nu=(5,4,2)$ and $\alpha=(3,1)$. Then, on the left, there is an example of a tableau that is not a Kronecker tableau, because it does not satisfy the $\alpha$--condition, and on the right, there is an example of a Kronecker tableau:
\begin{center}
\scalebox{0.9}{
\begin{tikzpicture}
\draw (0,0) rectangle (0.5,0.5);
\node at (0.25,0.25) {3};
\draw (0,0.5) rectangle (0.5,1);
\node at (0.25,0.75) {1};
\draw[fill=blue!40] (0,1) rectangle (0.5,1.5);
\draw[fill=blue!40] (0,1.5) rectangle (0.5,2);
\draw (0.5,0.5) rectangle (1,1);
\node at (0.75,0.75) {3};
\draw (0.5,1) rectangle (1,1.5);
\node at (0.75,1.25) {2};
\draw[fill=blue!40] (0.5,1.5) rectangle (1,2);
\draw (1,1) rectangle (1.5,1.5);
\node at (1.25, 1.25) {2};
\draw[fill=blue!40] (1,1.5) rectangle (1.5,2);
\draw (1.5,1.5) rectangle (2,2);
\node at (1.75, 1.75) {1};
\draw (2,1.5) rectangle (2.5, 2);
\node at (2.25, 1.75) {2};
\node at (1,-0.25) {No Kronecker tableau};

\draw (5,0) rectangle (5.5,0.5);
\node at (5.25,0.25) {3};
\draw (5,0.5) rectangle (5.5,1);
\node at (5.25,0.75) {1};
\draw[fill=blue!40] (5,1) rectangle (5.5,1.5);
\draw[fill=blue!40] (5,1.5) rectangle (5.5,2);
\draw (5.5,0.5) rectangle (6,1);
\node at (5.75,0.75) {2};
\draw (5.5,1) rectangle (6,1.5);
\node at (5.75,1.25) {1};
\draw[fill=blue!40] (5.5,1.5) rectangle (6,2);
\draw (6,1) rectangle (6.5,1.5);
\node at (6.25, 1.25) {3};
\draw[fill=blue!40] (6,1.5) rectangle (6.5,2);
\draw (6.5,1.5) rectangle (7,2);
\node at (6.75, 1.75) {2};
\draw (7,1.5) rectangle (7.5, 2);
\node at (7.25, 1.75) {2};
\node at (6.25,-0.25) {Kronecker tableau};
\end{tikzpicture}}
\end{center}
In both cases, the reverse reading word is a $(3,1)$--lattice permutation.

C. Ballantine and R. Orellana introduce the notion of the Kronecker tableaux to give a combinatorial description of a special kind of Kronecker coefficients. 
\begin{theorem}[C. Ballantine, R. Orellana, \cite{MR2264933}]\label{ThmKT}
Let $n$ and $p$ be positive integers such that $n \geq 2p$. Let $\lambda=(\lambda_1, \dots, \lambda_{\ell(\lambda)})$ and $\nu$ be partitions of $n$.
\begin{itemize}
\item[(a)] If $\lambda_1 \geq 2p-1$,   the multiplicity of $s_{\nu}$ in $
s_{(n-p,p)}*s_{\lambda}$ equals ${\displaystyle \sum_{\stackrel{\alpha \vdash p}{\alpha \subseteq
\lambda\cap \nu}}k_{\alpha\nu}^{\lambda}.}$ 
\item[(b)] If $\ell(\lambda) \geq 2p-1$, the multiplicity of
$s_{\nu}$ in $ s_{(n-p,p)}*s_{\lambda}$ equals ${\displaystyle \sum_{\stackrel{\alpha \vdash
p}{\alpha \subseteq \lambda^\prime \cap \nu^\prime}}k_{\alpha\nu^\prime}^{\lambda^\prime}}$, where $\lambda^\prime$ denotes the conjugate partition of $\lambda$.
\end{itemize}
\end{theorem}
The families that we study in this paper present some interesting stability properties. We start with examples of them. 

In Table \ref{table1} we present the values of $\overline{g}_{(k^a),(k^b)}^{(k)}$ that correspond to setting $a=b$ in Family 1. 

\begin{table}[h!]
\centering
\caption{Family 1: Case $b=a$, for $a=0,\dots, 5$}
\label{table1}
\begin{tabular}{c|rcccccccccccccc}
k & 0 & 1 & 2 & 3 & 4 & 5 & 6 & 7 & 8 & 9 & 10 & 11 & 12  & OEIS\footnotemark[1] \\ \hline
a= 0 &  1 & 0 & 0 & 0 & 0 & 0 & 0 & 0 & 0 & 0 & 0 & 0 & 0 & A000007 \\
a= 1 &  1 & 1 & 2 & 2 & 3 & 3 & 4 & 4 & 5 & 5 & 6 & 6 & 7 & A008619 \\
a= 2 &  1 & 1 & 3 & 4 & 7 & 9 & 14 & 17 & 24 & 29 & 38 & 45 & 57 & A266769 \\
a= 3 &  1 & 1 & 3 & 5 & 9 & 13 & 22 & 30 & 45 & 61 & 85 & 111 & 150 & A001993 \\
a= 4 & 1 & 1 & 3 & 5 & 10 & 15 & 26 & 38 & 60 & 85 & 125 & 172 & 243 & A070557 \\
a= 5 & 1 & 1 & 3 & 5 & 10 & 16 & 28 & 42 & 68 & 100 & 151 & 215 & 312 & A070558 
\end{tabular}
\end{table}
\footnotetext[1]{These references are taken out from the On--Line Encyclopedia of Integer Sequences, https://oeis.org/.}
We observe that the columns in Table \ref{table1} are stable sequences. According to the references in the On--Line Encyclopedia of Integer Sequences, the sequences appearing in Table \ref{table1} are related to Coxeter groups, braid groups and \emph{Molein series}, which are the generating functions attached to linear representations of a group on a finite dimensional vector space. 

Setting $b=a$ in Family 2, $\overline{g}^{(k)}_{\left( (k+i)^a \right) (k^a)} $, we obtain that after some initial zeros, the sequence defined by the non--zero reduced Kronecker coefficients is independent of $i$. In fact, it is equal to the sequence defined by Family 1 for $b=a$.  
Let see this phenomenon with an example that the reader should compare with Table \ref{table1}. 

\begin{table}[h!]	
\caption{Family 2: case $a=b=2$.}
\label{table2}
\begin{tabular}{c|rccccccccccccccccccc}
k & 0 & 1 & 2 & 3 & 4 & 5 & 6 & 7 & 8 & 9 & 10 & 11 & 12 & 13 & 14 & 15 & 16 \\ \hline
i= 0 & 1 & 1 & 3 & 4 & 7 & 9 & 14 & 17 & 24 & 29 & 38 & 45 & 57 & 66 & 81 & 93 & 111 \\
i= 1 & 0 & 0 & 0 & 1 & 1 & 3 & 4 & 7 & 9 & 14 & 17 & 24 & 29 & 38 & 45 & 57 & 66 \\
i= 2 & 0 & 0 & 0 & 0 & 0 & 0 & 1 & 1 & 3 & 4 & 7 & 9 & 14 &17 & 24 & 29 & 38  \\
i= 4 & 0 & 0 & 0 & 0 & 0 & 0 & 0 & 0 & 0 & 1 & 1 & 3 & 4 &  7 & 9 & 14 &17  \\
i= 5 & 0 & 0 & 0 & 0 & 0 & 0 & 0 & 0 & 0 & 0 & 0 & 0 & 1 & 1 & 3 & 4 &  7\\
\end{tabular}
\end{table}

Note that each row of Table \ref{table2} is exactly the shifting of the third row of Table \ref{table1} by a sequence of zeros. Note also that the columns are eventually constant with limit 0. 

Looking at the diagonals of Table \ref{table3}, we observe that Family 3 also contains stable sequences when $b=a+1$. 
\newpage
\begin{table}[h!]
\centering
\caption{Family 3: case $a=2$ and $b=3$.}
\label{table3}
\begin{tabular}{c|rccccccccccccccc}
k & 0 & 1 & 2 & 3 & 4 & 5 & 6 & 7 & 8 & 9 & 10 & 11 & 12  & 13 & OEIS \\ \hline
i= 0 &  \bf{1} & 1 & 3 & 4 & 7 & 9 & 14 & 17 & 24 & 29 & 38 & 45 & 57 & 66 & A266769 \\
i= 1 &  0 & \bf{1} & \bf{2} & 4 & 7 & 11 & 16 & 23 & 31 & 41 & 53 & 67 & 83 & 102& A000601 \\
i= 2 &  0 & 0 & \bf{1} & \bf{2} & \bf{5} & 8 & 14 & 20 & 30 & 40 & 55 & 70 & 91 & 112& A006918 \\
i= 3 &  0 & 0 & 0 & \bf{1} & \bf{2} & \bf{5} & \bf{9} & 15 & 23 & 34 & 47 & 64 & 84 & 108& A014126 \\
i= 4 &  0 & 0 & 0 & 0 & \bf{1} & \bf{2} & \bf{5} & \bf{9} & \bf{16} & 24 & 37 & 51 & 71 & 93&  \\
i= 5 & 0 & 0 & 0 & 0 & 0 & \bf{1} & \bf{2} & \bf{5} & \bf{9} & \bf{16} & \bf{25} & 38 & 54 & 75&  A175287 
\end{tabular}
\end{table}

We give the generating function of this sequence of stable values. For instance, in Table \ref{table3}, the resulting sequence is 1, 2, 5, 9, 16, 25, $\dots$ According to the references in the On--Line Encyclopedia of Integer Sequences, the sequences appearing in Table \ref{table3} are also related to Molein series.

Now we are ready to compute  the generating function for the families of reduced Kronecker coefficients that we consider. This translated immediately  to results for Kronecker coefficients, Corollary \ref{CorKC1} and Corollary \ref{CorKC2}.
\begin{theorem}\label{ThmGF}
We have the following generating functions for Families 1 and 2.
\begin{enumerate}
\item[\textbf{1.}] For $b=a$, the generating function of the reduced Kronecker coefficients of Family 1, $\overline{g}_{(k^a), (k^a)}^{(k)}$, is
\begin{eqnarray*}
\mathcal{F}_{a} = \frac{1}{(1-x)(1-x^2)^2 \cdots  (1-x^a)^2 (1-x^{a+1})}.
\end{eqnarray*}
\item[\textbf{2.}] For $b=a$, the generating function of the reduced Kronecker coefficients of Family 2, $\overline{g}_{((k+i)^a), (k^a)}^{(k)}$, for $k\geq \frac{a(a+1)}{2}\cdot i$ is also $\mathcal{F}_a$.
\item[\textbf{3.}] For $b\neq a$ in Families 1 and 2, the reduced Kronecker coefficients of these three families take values 0 or 1, depending on $k$ and $i$. 
\end{enumerate}
\end{theorem}
We present the proof of Theorem \ref{ThmGF} after the proof following result.

\begin{theorem}\label{ThmFam3}
The generating function for Family 3 is:
\begin{enumerate}
\item[\textbf{1.}]  For $b=a+1$, the stable value of the $j^{th}$ diagonal corresponds to the reduced Kronecker coefficients $\overline{g}^{(k)}_{(k^a), (2k-j,k^{a-1})}$, for $k\geq 2j$. Their generating function is 
\begin{eqnarray*}
\mathcal{G}_a = \frac{1}{(1-x)^2(1-x^2)^3\dots (1-x^{a-1})^3(1-x^a)^2(1-x^{a+1})}.
\end{eqnarray*}
In this case, we denote by $\overline{\overline{g}}_a(j)$ the stable value of the $j^{th}$ diagonal of the reduced Kronecker coefficients of Family 3, for $b=a+1$. 
\item[\textbf{2.}]  For $b\neq a+1$ in Family 3, the reduced Kronecker coefficients of this family take values 0 or 1, depending on $k$ and $i$. 
\end{enumerate}
\end{theorem}

\begin{proof}
For $b=a+1$, we look at the element of the $j^{th}$ diagonal. Fix positive integers $a$ and $j$. 
 The $j^{th}$ diagonal is describe by the coefficients $\overline{g}^{(k)}_{(k^a),(2k-j,k^{a-1}) }$, with $j=k-i\geq 0$ and $k\geq 2j$.
 
 By Corollary \ref{StabilityBound}, we can take $N=(a+4)k$, and these reduced Kronecker coefficients correspond to the following Kronecker coefficients:
\begin{eqnarray*}
\overline{g}^{(k)}_{(k^a)(2k-j,k^{a-1}) }= g^{((a+2)k,k)}_{(3k,k^a)(2k+j,2k-j,k^{a-1}) }.
\end{eqnarray*}
Applying Theorem \ref{ThmKT}, these coefficients count the Kronecker tableaux of shape $(3k,k^a)/\alpha$ and type $(2k+j,2k-j,k^{a-1})/\alpha$, with $\alpha \vdash k$ and $\ell(\alpha)\leq a+1$. 

It is immediate that the generating function $\mathcal{G}_a$ counts the coloured partitions with parts in $\mathcal{C}_a = \{1,\overline{1},2,\overline{2},\overline{\overline{2}},\dots, a-1,\overline{a-1},\overline{\overline{a-1}}, a,\overline{a},a+1\}$, for which the parts $i$, $\overline{i}$ and $\overline{\overline{i}}$ have weight $i$.  

To prove that Theorem \ref{ThmFam3} holds, we give a bijective map between coloured partitions with parts in $\mathcal{C}_a$ and  Kronecker tableaux of shape $(3k,k^a)/\alpha$ and type $(2k+j,2k-j,k^{a-1})/\alpha$, with $\alpha \vdash k$ and $\ell(\alpha)\leq a+1$.  

The bijection is defined by the following algorithm: to a coloured partition $\beta$ of $j$ with parts in $\mathcal{C}_a$, we associate a Kronecker tableau $T(\beta)$ as follows. 
First, we identify each element of $\mathcal{C}_a$ with a column of height $a+1$: \vspace{0.3cm}

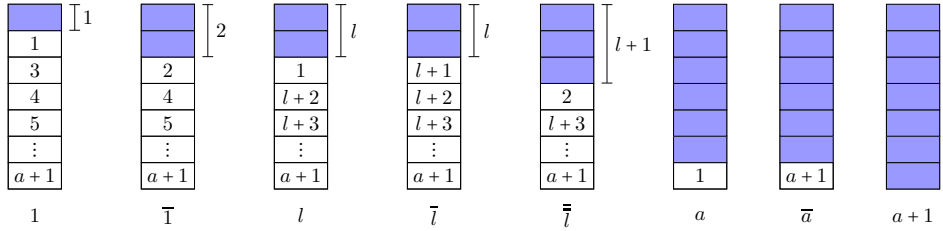
\scalebox{0.7}{
\begin{tikzpicture}
\draw (0,0) rectangle (1,3.5); 
\draw (0,0) rectangle (1,0.5); 
\draw (0,0.5) rectangle (1,1);
\draw (0,1) rectangle (1,1.5);
\draw (0,1.5) rectangle (1,2);
\draw (0,2) rectangle (1,2.5);
\draw (0,2.5) rectangle (1,3);
\draw[fill=blue!40] (0,3) rectangle (1,3.5);
\node at (0.5,-0.5) {$1$};
\node at (0.5,0.25) {$a+1$};
\node at (0.5,0.75) {$\vdots$};
\node at (0.5,1.25) {$5$};
\node at (0.5,1.75) {$4$};
\node at (0.5,2.25) {$3$};
\node at (0.5,2.75) {$1$};
\draw[|-|] (1.25,3) --(1.25,3.5);
\node at (1.5, 3.25) {$1$};

\draw (2.5,0) rectangle (3.5,3.5); 
\draw (2.5,0) rectangle (3.5,0.5); 
\draw (2.5,0.5) rectangle (3.5,1);
\draw (2.5,1) rectangle (3.5,1.5);
\draw (2.5,1.5) rectangle (3.5,2);
\draw (2.5,2) rectangle (3.5,2.5);
\draw[fill=blue!40] (2.5,2.5) rectangle (3.5,3);
\draw[fill=blue!40] (2.5,3) rectangle (3.5,3.5);
\node at (3,-0.5) {$\overline{1}$};
\node at (3,0.25) {$a+1$};
\node at (3,0.75) {$\vdots$};
\node at (3,1.25) {$5$};
\node at (3,1.75) {$4$};
\node at (3,2.25) {$2$};
\draw[|-|] (3.75,2.5) --(3.75,3.5);
\node at (4, 3) {$2$};

\draw (5,0) rectangle (6,3.5); 
\draw (5,0) rectangle (6,0.5); 
\draw (5,0.5) rectangle (6,1);
\draw (5,1) rectangle (6,1.5);
\draw (5,1.5) rectangle (6,2);
\draw (5,2) rectangle (6,2.5);
\draw[fill=blue!40] (5,2.5) rectangle (6,3);
\draw[fill=blue!40] (5,3) rectangle (6,3.5);
\node at (5.5,-0.5) {$l$};
\node at (5.5,0.25) {$a+1$};
\node at (5.5,0.75) {$\vdots$};
\node at (5.5,1.25) {$l+3$};
\node at (5.5,1.75) {$l+2$};
\node at (5.5,2.25) {$1$};
\draw[|-|] (6.25,2.5) --(6.25,3.5);
\node at (6.5, 3) {$l$};

\draw (7.5,0) rectangle (8.5,3.5); 
\draw (7.5,0) rectangle (8.5,0.5); 
\draw (7.5,0.5) rectangle (8.5,1);
\draw (7.5,1) rectangle (8.5,1.5);
\draw (7.5,1.5) rectangle (8.5,2);
\draw (7.5,2) rectangle (8.5,2.5);
\draw[fill=blue!40] (7.5,2.5) rectangle (8.5,3);
\draw[fill=blue!40] (7.5,3) rectangle (8.5,3.5);
\node at (8,-0.5) {$\overline{l}$};
\node at (8,0.25) {$a+1$};
\node at (8,0.75) {$\vdots$};
\node at (8,1.25) {$l+3$};
\node at (8,1.75) {$l+2$};
\node at (8,2.25) {$l+1$};
\draw[|-|] (8.75,2.5) --(8.75,3.5);
\node at (9, 3) {$l$};

\draw (10,0) rectangle (11,3); 
\draw (10,0) rectangle (11,0.5); 
\draw (10,0.5) rectangle (11,1);
\draw (10,1) rectangle (11,1.5);
\draw (10,1.5) rectangle (11,2);
\draw[fill=blue!40] (10,2) rectangle (11,2.5);
\draw[fill=blue!40] (10,2.5) rectangle (11,3);
\draw[fill=blue!40] (10,3) rectangle (11,3.5);
\node at (10.5,-0.5) {$\overline{\overline{l}}$};
\node at (10.5,0.25) {$a+1$};
\node at (10.5,0.75) {$\vdots$};
\node at (10.5,1.25) {$l+3$};
\node at (10.5,1.75) {$2$};
\draw[|-|] (11.25,2) --(11.25,3.5);
\node at (11.75, 2.75) {$l+1$};

\draw (12.5,0) rectangle (13.5,3); 
\node at (13,0.25) {$1$};
\draw(12.5,0) rectangle (13.5,0.5); 
\draw[fill=blue!40] (12.5,0.5) rectangle (13.5,1);
\draw[fill=blue!40] (12.5,1) rectangle (13.5,1.5);
\draw[fill=blue!40] (12.5,1.5) rectangle (13.5,2);
\draw[fill=blue!40] (12.5,2) rectangle (13.5,2.5);
\draw[fill=blue!40] (12.5,2.5) rectangle (13.5,3);
\draw[fill=blue!40] (12.5,3) rectangle (13.5,3.5);
\node at (13,-0.5) {$a$};

\draw (14.5,0) rectangle (15.5,3); 
\node at (15,0.25) {$a+1$};
\draw (14.5,0) rectangle (15.5,0.5); 
\draw[fill=blue!40] (14.5,0.5) rectangle (15.5,1);
\draw[fill=blue!40] (14.5,1) rectangle (15.5,1.5);
\draw[fill=blue!40] (14.5,1.5) rectangle (15.5,2);
\draw[fill=blue!40] (14.5,2) rectangle (15.5,2.5);
\draw[fill=blue!40] (14.5,2.5) rectangle (15.5,3);
\draw[fill=blue!40] (14.5,3) rectangle (15.5,3.5);
\node at (15,-0.5) {$\overline{a}$};

\draw[fill=blue!40] (16.5,0) rectangle (17.5,3); 
\draw[fill=blue!40] (16.5,0) rectangle (17.5,0.5); 
\draw[fill=blue!40] (16.5,0.5) rectangle (17.5,1);
\draw[fill=blue!40] (16.5,1) rectangle (17.5,1.5);
\draw[fill=blue!40] (16.5,1.5) rectangle (17.5,2);
\draw[fill=blue!40] (16.5,2) rectangle (17.5,2.5);
\draw[fill=blue!40] (16.5,2.5) rectangle (17.5,3);
\draw[fill=blue!40] (16.5,3) rectangle (17.5,3.5);
\node at (17,-0.5) {$a+1$};
\end{tikzpicture}}

for $l=2, \dots, a-1$. 
The partition $\alpha$ of $k$ is defined by counting all shaded boxes on the Kronecker tableau $T(\beta)$. If we consider $\beta=(\overline{1}^j)$, we have $2j$ shaded boxes. Then, $k\geq 2j$. We consider partitions $\alpha$ of $k$. Note that for $\beta=(1^j)$, there are only $j$ shaded boxes, which is not enough to obtain $\alpha$. 
That is why, the next step is to add as many columns as shaded boxes we need in order to obtain a partition of $k$. Denote by $m_l$, with $l\in \mathcal{C}_a$, the number of times that the part $l$ appears in $\beta$. Until that moment, we have the following number of shaded boxes
\begin{eqnarray*}
N_\beta= \sum_{l=1}^{a+1}l\cdot m_l + \sum_{l=2}^a l\cdot m_{\overline{l}} + 2m_{\overline{1}} + \sum_{l=2}^{a-1} (l+1)\cdot m_{\overline{\overline{l}}}.
\end{eqnarray*} 
We also know that $\beta$ is a coloured partition of $j$,
\begin{eqnarray}\label{j3}
j=\sum_{l=1}^{a+1}l\cdot m_l + \sum_{l=1}^a l\cdot m_{\overline{l}} + \sum_{l=2}^{a-1} l\cdot m_{\overline{\overline{l}}}.
\end{eqnarray}
Let us define $m_0$ as $k$ minus the number of shaded boxes that we already have, $N_\beta$. Then,
\begin{eqnarray*}
m_0 = k-j-m_{\overline{1}} - \sum_{l=2}^{a-1} m_{\overline{\overline{l}}}.
\end{eqnarray*}
At this point, we add $m_0$ copies of the following column of height $a+1$: \vspace{0.25cm}
\begin{center}
\scalebox{0.8}{
\begin{tikzpicture}
\draw (0,0) rectangle (1,3.5); 
\draw (0,0) rectangle (1,0.5); 
\draw (0,0.5) rectangle (1,1);
\draw (0,1) rectangle (1,1.5);
\draw (0,1.5) rectangle (1,2);
\draw (0,2) rectangle (1,2.5);
\draw (0,2.5) rectangle (1,3);
\draw[fill=blue!40] (0,3) rectangle (1,3.5);
\node at (0.5,0.25) {$a+1$};
\node at (0.5,0.75) {$\vdots$};
\node at (0.5,1.25) {$5$};
\node at (0.5,1.75) {$4$};
\node at (0.5,2.25) {$3$};
\node at (0.5,2.75) {$2$};
\draw[|-|] (1.25,3) --(1.25,3.5);
\node at (1.5, 3.25) {$1$};
\end{tikzpicture}}
\end{center}
Then, $\alpha$ is the partition defined by the shaded boxes, as follows:
\begin{eqnarray}\label{alpha3}
\alpha_{a+1} &=& m_{a+1}, \nonumber \\
\alpha_l &=& \alpha_{l+1} +m_l + m_{\overline{l}} + m_{\overline{\overline{l-1}}}, \hspace{0.5cm} \text{for }l=3,\dots, a, \\
\alpha_2 &=& \alpha_3 + m_2 + m_{\overline{2}} + m_{\overline{1}} , \nonumber \\
\alpha_1 &=& \alpha_2 +m_1 +m_0.  \nonumber 
\end{eqnarray}
Note that there exists only a way to order all these columns in such a way that they form a \ssyt.
These columns correspond to the first columns on the left side of $T(\beta)$.

Before continuing with the algorithm, let see an example: take $a=3$, $j=3$, $k=7$ and the coloured partition $\beta=(\overline{\overline{2}},1)$. Then, the corresponding Kronecker tableau obtained by our algorithm is
\begin{center}
\scalebox{0.8}{
\begin{tikzpicture}
\draw[fill=blue!40] (0,0.5)--(0.5,0.5)--(0.5,1.5)--(2.5,1.5)--(2.5,2)--
(0,2)--(0,0.5);
\draw (0,0) rectangle (3.5,2);
\draw (0,1.5) rectangle (10.5,2);

\draw (0,0.5)--(3.5,0.5);
\draw (0,1)--(3.5,1);
\draw (0,1.5)--(3.5,1.5);
\draw (0.5,0) --(0.5,2);
\draw (1,0) --(1,2);
\draw (1.5,0) --(1.5,2);
\draw (2,0) --(2,2);
\draw (2.5,0) --(2.5,2);
\draw (3,0) --(3,2);
\draw (3.5,0) --(3.5,2);
\draw (4,1.5) --(4,2);
\draw (4.5,1.5)--(4.5,2);
\draw (5,1.5) --(5,2);
\draw (5.5,1.5)--(5.5,2);
\draw (6,1.5) --(6,2);
\draw (6.5,1.5)--(6.5,2);
\draw (7,1.5) --(7,2);
\draw (7.5,1.5)--(7.5,2);
\draw (8,1.5) --(8,2);
\draw (8.5,1.5)--(8.5,2);
\draw (9,1.5) --(9,2);
\draw (9.5,1.5)--(9.5,2);
\draw (10,1.5) --(10,2);
\draw (10.5,1.5)--(10.5,2);

\node at (0.25,0.25) {2};
\node at (0.75,0.25) {4};
\node at (1.25,0.25) {4};
\node at (1.75,0.25) {4};
\node at (2.25,0.25) {4};
\node at (2.75,0.25) {4};
\node at (3.25,0.25) {4};

\node at (0.75,0.75) {3};
\node at (1.25,0.75) {3};
\node at (1.75,0.75) {3};
\node at (2.25,0.75) {3};
\node at (2.75,0.75) {3};
\node at (3.25,0.75) {3};

\node at (0.75,1.25) {1};
\node at (1.25,1.25) {2};
\node at (1.75,1.25) {2};
\node at (2.25,1.25) {2};
\node at (2.75,1.25) {2};
\node at (3.25,1.25) {2};

\node at (2.75,1.75) {1};
\node at (3.25,1.75) {1};
\node at (3.75,1.75) {1};
\node at (4.25,1.75) {1};
\node at (4.75,1.75) {1};
\node at (5.25,1.75) {1};
\node at (5.75,1.75) {1};
\node at (6.25,1.75) {1};
\node at (6.75,1.75) {1};
\node at (7.25,1.75) {1};
\node at (7.75,1.75) {1};
\node at (8.25,1.75) {2};
\node at (8.75,1.75) {2};
\node at (9.25,1.75) {2};
\node at (9.75,1.75) {2};
\node at (10.25,1.75) {4};
\end{tikzpicture}}
\end{center}
To finish the construction of $T(\beta)$, we proceed as follows: the $l^{th}$ row is filled with $l$, for $l=2,\dots a+1$, and the first row is filled with the remaining numbers of the type $(2k+j,2k-j,k^{a-1})/\alpha$ in weakly increasing order from left to right.

Let us check that $T(\beta)$ is a Kronecker tableau, and hence, that the map is well--defined. 
\begin{itemize}
\item By construction, $T(\beta)$ is a \ssyt of shape $(3k+i,k^a)/\alpha$ and type $(2k+j,2k-j,k^{a-1})/\alpha$, where $\alpha$ is a partition of $k$ and $\ell(\alpha)\leq a+1$.

\item The sequence $\alpha$ defined by counting the shaded boxes is a partition of $k$. By the recurrence \eqref{alpha3} that describes $\alpha$, the sequence is clearly a partition. Let us see that the sum of its parts is $k$. We express $\alpha$ in terms of $m_l$, with $l\in\mathcal{C}_a \cup \{0 \}$. 
\begin{eqnarray*}
\alpha_{a+1} &=& m_{a+1} , \\
\alpha_l &=& \sum_{k=l}^{a+1} m_k + \sum_{k=l}^{a} m_{\overline{k}} + \sum_{k=l}^{a-1} m_{\overline{\overline{k-1}}}, \hspace{0.5cm} \text{for }l=3,\dots, a, \\
\alpha_2 &=& \sum_{k=2}^{a+1} m_k + \sum_{k=1}^{a} m_{\overline{k}} + \sum_{k=2}^{a-1} m_{\overline{\overline{k}}} , \hspace{0.5cm} \text{and, } \\
\alpha_1 &=& k-j + \sum_{k=1}^{a+1} m_k + \sum_{k=2}^ a m_{\overline{k}},
\end{eqnarray*}
where we use directly the definition of $m_0$ to compute $\alpha_1$. 
Then, 
\begin{multline*}
\sum_{l=1}^{a+1} \alpha_l = m_{a+1} + \sum_{l=3}^a \left( \sum_{k=l}^{a+1} m_k + \sum_{k=l}^{a} m_{\overline{k}} + \sum_{k=l}^{a-1} m_{\overline{\overline{k-1}}} \right) + \sum_{k=2}^{a+1} m_k + \sum_{k=1}^{a} m_{\overline{k}} + \sum_{k=2}^{a-1} m_{\overline{\overline{k}}} +\\ 
+ k-j + \sum_{k=1}^{a+1} m_k + \sum_{k=2}^ a m_{\overline{k}} 
= k-j +\sum_{l=1}^{a+1} l\cdot m_l + \sum_{l=1}^a l\cdot m_{\overline{l}} + \sum_{l=2}^{a-1}l\cdot m_{\overline{\overline{l}}} = k .
\end{multline*}

\item We need to check also that $(\# 1)_{R1} \geq k- \alpha_1$. Otherwise, we have a column that we cannot fill.

We count $(\# 1)_{R1}$ as the total number of $1'$s minus the \fbox{1} boxes in all rows different from the first one, i.e. $(\# 1)_{R1} = 2k+j-\alpha_1 - \sum_{l=1}^a m_{l}$.
Then, by \eqref{j3}, $k+j -\sum_{l=1}^a m_l \geq 0$.

\item The reverse word is an $\alpha$--lattice permutation. 

The reverse reading word of $T(\beta)$ is of the form
\begin{multline*}
\left(\# a+1\right)_{R1}\dots \left(\# 1\right)_{R1} \left(\# 2\right)_{R2} \left(\# 1\right)_{R2} \dots \\ \dots \left(\# l\right)_{Rl} \left(\#2\right)_{Rl} \left(\# 1\right)_{Rl} \dots \left(\# a+1\right)_{Ra+1} \left(\# 2\right)_{Ra+1} \left(\# 1\right)_{Ra+1}.
\end{multline*}
We proceed to check that this sequence is an $\alpha$--lattice permutation following the sequence from left to right.  
\begin{itemize}
\item At the level of the first row, we check that $\alpha_{l+1} +(\# l+1)_{R1} \leq \alpha_l$.

For $l=3,\dots,a$, we have that $(\#l+1)_{R1} = (\#1)_{Rl+1} + (\#2)_{Rl+1}$. We take the count as the total number of $(l+1)'$s minus the number of \fbox{$l+1$} boxes of the $(l+1)^{th}$. Then,
\begin{eqnarray*}
\alpha_{l+1} + (\# l+1)_{R1} = \alpha_{l+1} +m_l+m_{\overline{\overline{l-1}}} \leq \alpha_{l+1} +m_l+m_{\overline{\overline{l-1}}} + m_{\overline{l}} = \alpha_l.
\end{eqnarray*}
For $l=2$, counting in the same way, we have that $(\#3)_{R1}= m_2+m_{\overline{1}}$. Then,
\begin{eqnarray*}
\alpha_3 + (\#3)_{R1} = \alpha_2+m_2+m_{\overline{1}} \leq \alpha_2+m_2+m_{\overline{2}}+ m_{\overline{1}} = \alpha_2.
\end{eqnarray*}
For $l=1$, we count $(\#2)_{R1}$.
\begin{eqnarray}\label{alphacond3}
(\#2)_{R1} = \underbrace{2k-j-\alpha_2}_{\text{total}}- \underbrace{(m_0+k-\alpha_1)}_{2^{nd}\text{row}} -\underbrace{m_{\overline{1}} - \sum_{l=2}^{a-1}m_{\overline{\overline{l}}}}_{\text{other rows}} = \alpha_1-\alpha_2.
\end{eqnarray}
Then, $\alpha_2+ (\#2)_{R1} = \alpha_1$. 

\item At the level of the second row, we check that $\alpha_2+ (\#2)_{R1}+(\#2)_{R2} = \alpha_1+(\#1)_{R1}$.

By \eqref{alphacond3}, we only need to check that $(\#2)_{R2}\leq (\#1)_{R1}$. Since $(\#2)_{R2}=m_0+k-\alpha_1$ and $(\#1)_{R1} \geq 3k-\alpha_1$, by the definition of $m_0$, we have that
$2k-m_0 =k+j+\sum_{l=2}^{a-1}m_{\overline{\overline{l}}}+m_{\overline{1}} \geq 0$. This implies that
\begin{eqnarray*}
(\#2)_{R2} = m_0+k-\alpha_1 \leq 3k-\alpha_1 \leq (\#1)_{R1}.
\end{eqnarray*}  

\item At the level of the $l^{th}$ row, for $l=3,\dots,a$, we have that
\begin{eqnarray*}
\alpha_{l+1} + (\#l+1)_{R1} + (\#l+1)_{Rl+1} = \alpha_{l} + (\#l)_{R1} + (\#l)_{Rl},
\end{eqnarray*}
since the left--hand side is exactly the total number of $(l+1)$ plus $\alpha_{l+1}$ and the right--hand side is exactly the total number of $l$ plus $\alpha_{l}$. Thus, both sides are equal to $k$.

For $l=2$, we check that $\alpha_{3} + (\#3)_{R1} + (\#3)_{R3} \leq  \alpha_{2} + (\#2)_{R1} + (\#2)_{R2}$. The left--hand side is exactly the total number of $3'$s plus $\alpha_3$, i.e. $k$. The right--hand side is $k+m_0$, using \eqref{alphacond3} and that $(\#2)_{R2} =m_0+k-\alpha_1$. 

Finally, we check all the inequalities related to the boxes \fbox{2} and \fbox{1} involving more rows than the first and the second ones, i.e, for $s=3,\dots,a$, we check that $ \displaystyle{\alpha_2 + \sum_{l=1}^{s} (\#2)_{Rl} \leq \alpha_1 + \sum_{l=1}^{s-1} (\#1)_{Rl}}$.
We have that
\begin{eqnarray*}
\sum_{l=1}^{s} (\#2)_{Rl} &=& 2k-j-\alpha_2 - \sum_{l=s+1}^{a+1}(\#2)_{Rl} = 2k-j-\alpha_2 - \sum_{l=s}^{a-1} m_{\overline{\overline{l}}}, \\
\sum_{l=1}^{s-1} (\#1)_{Rl} &=& 2k+j-\alpha_1 - \sum_{l=s}^{a+1}(\#1)_{Rl} = 2k+j-\alpha_1 - \sum_{l=s}^{a} m_l.
\end{eqnarray*}
Then, the inequality $\displaystyle{\alpha_2 + \sum_{l=1}^{s} (\#2)_{Rl} \leq \alpha_1 + \sum_{l=1}^{s-1}}$ follows from \eqref{j3} and the fact that $\displaystyle{2j + \sum_{l=s}^{a-1}m_{\overline{\overline{l}}} - \sum_{l=s}^a m_l \geq 0}$.
\end{itemize}

\item For the $\alpha$--condition, if $\alpha_1=\alpha_2$, we have nothing to prove, and if $\alpha_1>\alpha_2$, we have that $(\#2)_{R1} = \alpha_1-\alpha_2$ by \eqref{alphacond3}.
\end{itemize}
Then, the \ssyt\ $T(\beta)$ defined by the algorithm is a Kronecker tableau and the map is well--defined and injective. 

Finally, we show that the map is also surjective. 
Consider a Kronecker tableau $T$ with shape $\lambda/\alpha=(3k,k^a)/\alpha$, type $\nu/\alpha=(2k+j,2k-j,k^{a-1})/\alpha$ and $\alpha$ a partition of $k$, with $\ell(\alpha)\leq a+1$. We will define the associated coloured partition $\beta$ of $j$ with parts in $\mathcal{C}_a$.
We start studying our initial Kronecker tableau, $T$. It has the following form
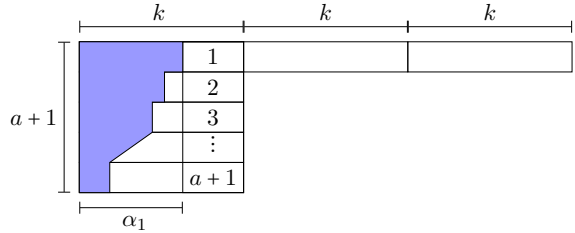
\begin{center}
\scalebox{0.8}{
\begin{tikzpicture}
\draw (1.5,1) rectangle (4.2,3.5);
\draw (4.2,3) rectangle (6.9,3.5);
\draw (6.9,3) rectangle (9.6,3.5);
\draw (1.5,1) rectangle (4.2,1.5);
\draw (1.5,1.5) rectangle (4.2,2);
\draw (1.5,2) rectangle (4.2,2.5);
\draw (1.5,2.5) rectangle (4.2,3);
\draw (1.5,3) rectangle (4.2,3.5);
\draw (3.2,1) rectangle (4.2,3.5);

\draw[fill=blue!40] (1.5,1) --(2,1) --(2,1.5) --(2.7,2) --(2.7,2.5) --(2.9,2.5) --(2.9,3) --(3.2,3) --(3.2,3.5) --(1.5,3.5)--(1.5,1);
\draw[|-|] (1.5,0.75) --(3.2,0.75);
\draw[|-|] (1.25,1) --(1.25,3.5);
\draw[|-|] (1.5,3.75) --(4.2, 3.75);
\draw[|-|] (4.2,3.75) --(6.9,3.75);
\draw[|-|] (6.9,3.75) --(9.6,3.75);

\node at (2.41,0.5) {$\alpha_1$};
\node at (0.75,2.25) {$a+1$};
\node at (2.81,4) {$k$};
\node at (5.55,4) {$k$};
\node at (8.25,4) {$k$};
\node at (3.7,1.25) {$a+1$};
\node at (3.7,1.8) {$\vdots$};
\node at (3.7,2.25) {$3$};
\node at (3.7,2.75) {$2$};
\node at (3.7,3.25) {$1$};
\end{tikzpicture}}
\end{center}
Excluding all columns of height 1 and of height $a+1$ with no shaded boxes, we claim that the following list summarizes all possible columns that appear in the remaining part of $T$.
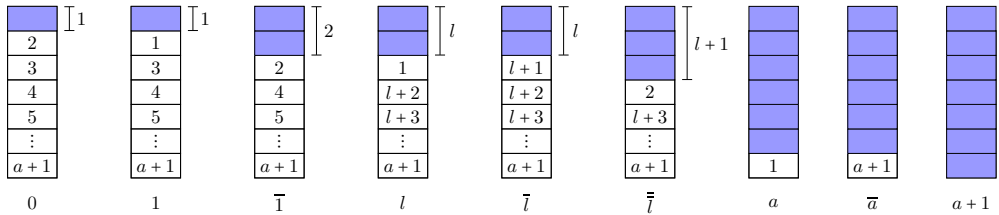
\begin{center}
\scalebox{0.65}{
\begin{tikzpicture}
\draw (-2.5,0) rectangle (-1.5,3.5); 
\draw (-2.5,0) rectangle (-1.5,0.5); 
\draw (-2.5,0.5) rectangle (-1.5,1);
\draw (-2.5,1) rectangle (-1.5,1.5);
\draw (-2.5,1.5) rectangle (-1.5,2);
\draw (-2.5,2) rectangle (-1.5,2.5);
\draw (-2.5,2.5) rectangle (-1.5,3);
\draw[fill=blue!40] (-2.5,3) rectangle (-1.5,3.5);
\node at (-2,0.25) {$a+1$};
\node at (-2,0.75) {$\vdots$};
\node at (-2,1.25) {$5$};
\node at (-2,1.75) {$4$};
\node at (-2,2.25) {$3$};
\node at (-2,2.75) {$2$};
\draw[|-|] (-1.25,3) --(-1.25,3.5);
\node at (-1, 3.25) {$1$};
\node at (-2,-0.5) {0};

\draw (0,0) rectangle (1,3.5); 
\draw (0,0) rectangle (1,0.5); 
\draw (0,0.5) rectangle (1,1);
\draw (0,1) rectangle (1,1.5);
\draw (0,1.5) rectangle (1,2);
\draw (0,2) rectangle (1,2.5);
\draw (0,2.5) rectangle (1,3);
\draw[fill=blue!40] (0,3) rectangle (1,3.5);
\node at (0.5,-0.5) {$1$};
\node at (0.5,0.25) {$a+1$};
\node at (0.5,0.75) {$\vdots$};
\node at (0.5,1.25) {$5$};
\node at (0.5,1.75) {$4$};
\node at (0.5,2.25) {$3$};
\node at (0.5,2.75) {$1$};
\draw[|-|] (1.25,3) --(1.25,3.5);
\node at (1.5, 3.25) {$1$};

\draw (2.5,0) rectangle (3.5,3.5); 
\draw (2.5,0) rectangle (3.5,0.5); 
\draw (2.5,0.5) rectangle (3.5,1);
\draw (2.5,1) rectangle (3.5,1.5);
\draw (2.5,1.5) rectangle (3.5,2);
\draw (2.5,2) rectangle (3.5,2.5);
\draw[fill=blue!40] (2.5,2.5) rectangle (3.5,3);
\draw[fill=blue!40] (2.5,3) rectangle (3.5,3.5);
\node at (3,-0.5) {$\overline{1}$};
\node at (3,0.25) {$a+1$};
\node at (3,0.75) {$\vdots$};
\node at (3,1.25) {$5$};
\node at (3,1.75) {$4$};
\node at (3,2.25) {$2$};
\draw[|-|] (3.75,2.5) --(3.75,3.5);
\node at (4, 3) {$2$};

\draw (5,0) rectangle (6,3.5); 
\draw (5,0) rectangle (6,0.5); 
\draw (5,0.5) rectangle (6,1);
\draw (5,1) rectangle (6,1.5);
\draw (5,1.5) rectangle (6,2);
\draw (5,2) rectangle (6,2.5);
\draw[fill=blue!40] (5,2.5) rectangle (6,3);
\draw[fill=blue!40] (5,3) rectangle (6,3.5);
\node at (5.5,-0.5) {$l$};
\node at (5.5,0.25) {$a+1$};
\node at (5.5,0.75) {$\vdots$};
\node at (5.5,1.25) {$l+3$};
\node at (5.5,1.75) {$l+2$};
\node at (5.5,2.25) {$1$};
\draw[|-|] (6.25,2.5) --(6.25,3.5);
\node at (6.5, 3) {$l$};

\draw (7.5,0) rectangle (8.5,3.5); 
\draw (7.5,0) rectangle (8.5,0.5); 
\draw (7.5,0.5) rectangle (8.5,1);
\draw (7.5,1) rectangle (8.5,1.5);
\draw (7.5,1.5) rectangle (8.5,2);
\draw (7.5,2) rectangle (8.5,2.5);
\draw[fill=blue!40] (7.5,2.5) rectangle (8.5,3);
\draw[fill=blue!40] (7.5,3) rectangle (8.5,3.5);
\node at (8,-0.5) {$\overline{l}$};
\node at (8,0.25) {$a+1$};
\node at (8,0.75) {$\vdots$};
\node at (8,1.25) {$l+3$};
\node at (8,1.75) {$l+2$};
\node at (8,2.25) {$l+1$};
\draw[|-|] (8.75,2.5) --(8.75,3.5);
\node at (9, 3) {$l$};

\draw (10,0) rectangle (11,3); 
\draw (10,0) rectangle (11,0.5); 
\draw (10,0.5) rectangle (11,1);
\draw (10,1) rectangle (11,1.5);
\draw (10,1.5) rectangle (11,2);
\draw[fill=blue!40] (10,2) rectangle (11,2.5);
\draw[fill=blue!40] (10,2.5) rectangle (11,3);
\draw[fill=blue!40] (10,3) rectangle (11,3.5);
\node at (10.5,-0.5) {$\overline{\overline{l}}$};
\node at (10.5,0.25) {$a+1$};
\node at (10.5,0.75) {$\vdots$};
\node at (10.5,1.25) {$l+3$};
\node at (10.5,1.75) {$2$};
\draw[|-|] (11.25,2) --(11.25,3.5);
\node at (11.75, 2.75) {$l+1$};

\draw (12.5,0) rectangle (13.5,3); 
\node at (13,0.25) {$1$};
\draw(12.5,0) rectangle (13.5,0.5); 
\draw[fill=blue!40] (12.5,0.5) rectangle (13.5,1);
\draw[fill=blue!40] (12.5,1) rectangle (13.5,1.5);
\draw[fill=blue!40] (12.5,1.5) rectangle (13.5,2);
\draw[fill=blue!40] (12.5,2) rectangle (13.5,2.5);
\draw[fill=blue!40] (12.5,2.5) rectangle (13.5,3);
\draw[fill=blue!40] (12.5,3) rectangle (13.5,3.5);
\node at (13,-0.5) {$a$};

\draw (14.5,0) rectangle (15.5,3); 
\node at (15,0.25) {$a+1$};
\draw (14.5,0) rectangle (15.5,0.5); 
\draw[fill=blue!40] (14.5,0.5) rectangle (15.5,1);
\draw[fill=blue!40] (14.5,1) rectangle (15.5,1.5);
\draw[fill=blue!40] (14.5,1.5) rectangle (15.5,2);
\draw[fill=blue!40] (14.5,2) rectangle (15.5,2.5);
\draw[fill=blue!40] (14.5,2.5) rectangle (15.5,3);
\draw[fill=blue!40] (14.5,3) rectangle (15.5,3.5);
\node at (15,-0.5) {$\overline{a}$};

\draw[fill=blue!40] (16.5,0) rectangle (17.5,3); 
\draw[fill=blue!40] (16.5,0) rectangle (17.5,0.5); 
\draw[fill=blue!40] (16.5,0.5) rectangle (17.5,1);
\draw[fill=blue!40] (16.5,1) rectangle (17.5,1.5);
\draw[fill=blue!40] (16.5,1.5) rectangle (17.5,2);
\draw[fill=blue!40] (16.5,2) rectangle (17.5,2.5);
\draw[fill=blue!40] (16.5,2.5) rectangle (17.5,3);
\draw[fill=blue!40] (16.5,3) rectangle (17.5,3.5);
\node at (17,-0.5) {$a+1$};
\end{tikzpicture}}
\end{center}
for  $l=2,\dots,a-1$.

Let us prove that there are no other kinds of columns of height $a+1$ with shaded boxes:
\begin{itemize}
\item For $l=3,\dots,a+1$, the box \fbox{$l$} cannot appear in any row $s\leq l-1$. Suppose that there is a column of the form
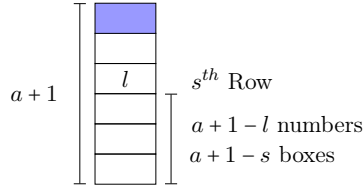
\begin{center}
\scalebox{0.8}{
\begin{tikzpicture}
\draw (1.5,0.5) rectangle (2.5,1);
\draw (1.5,1) rectangle (2.5,1.5);
\draw (1.5,1.5) rectangle (2.5,2);
\draw (1.5,2) rectangle (2.5,2.5);
\draw (1.5,2.5) rectangle (2.5,3);
\draw[fill=blue!40] (1.5,3) rectangle (2.5,3.5);
\draw[|-|] (1.25,0.5) --(1.25,3.5);
\draw[|-|] (2.75,0.5) --(2.75,2);
\node at (0.5,2) {$a+1$};
\node at (2,2.25) {$l$};
\node at (3.75,2.25) {$s^{th}$ Row};
\node at (4.5, 1.5) {$a+1-l$ numbers};
\node at (4.25,1) {$a+1-s$ boxes};
\end{tikzpicture} }
\end{center}
Since $T$ is a \ssyt\ we cannot fill the $a+1-s$ boxes of this column with different numbers, because we only have $a+1-l$ possibilities. 

\item For $l=3,\dots,a$, the box \fbox{$l$} cannot neither appear in any row $s\geq l+1$. We proceed by induction from $a$ to 3.

Look at the end of the reverse reading word
\begin{eqnarray*}
\left(\# a+1\right)_{R1} + \alpha_{a+1} \leq \alpha_a + (\# a)_{R1} + (\# a)_{Ra}.
\end{eqnarray*}
The left--hand side is exactly $k$ and there are $k-\alpha_a$ \fbox{$a$} boxes in total. Then, there are no more boxes than in the first row and in the $a^{th}$ row. 
Let us see that there cannot be \fbox{$l$} boxes in the $(l+1)^{th}$ row, assuming that in any $s^{th}$ row, with $s \geq l+2$, there are only \fbox{$1$} and \fbox{$s$} boxes. The part of the reverse reading word corresponding to the $(l+1)^{th}$ row says that
\begin{eqnarray*}
\left(\# l+1\right)_{R1} + \left(\# l+1\right)_{Rl+1} + \alpha_{l+1} \leq \alpha_l + (\# l)_{R1} + (\# i)_{Rl}.
\end{eqnarray*}
Applying the induction hypothesis, the left--hand side is exactly $k$. Then, $(\# l)_{Rl} + (\# l)_{Rl} \geq k-\alpha_l$. Since there are $k-\alpha_l$ \fbox{$l$} boxes in total, we also have that $(\# l)_{R1} + (\# l)_{Rl} \leq k-\alpha_l$, and there are no \fbox{$l$} boxes in any row different from the first and the $l^{th}$ ones.

\item The boxes \fbox{1} and \fbox{2} can appear in all the rows, since there are $2k+j-\alpha_1$ \fbox{1} boxes and $2k-j-\alpha_2$ \fbox{2} boxes. Both amounts are bigger than $k-\alpha_l$ for all $l=3,\dots a+1$, and there is no contradiction with the condition of the reverse reading word. 
\end{itemize}

To define $\beta$, we denote by $n_l$, with $l\in \mathcal{C}_a\cup\{0\}$, the number of occurrences of the column $l$ in $T$. Then,
\begin{eqnarray*}
\beta := \left( 1^{n_1} \overline{1}^{n_{\overline{1}}} 2^{n_2} \overline{2}^{n_{\overline{2}}} \overline{\overline{2}}^{n_{\overline{\overline{2}}}}\dots \overline{\overline{a-1}}^{n_{\overline{\overline{a-1}}}} a^{n_a}\overline{a}^{n_{\overline{a}}} a+1^{n_{a+1}} \right).
\end{eqnarray*}
We finish the proof showing that the sequence $\beta$ is a coloured partition of $j$. 
We separate into cases due to the $\alpha$--condition:
\begin{itemize}
\item If $\alpha_1=\alpha_2$, there are no \fbox{2} boxes in the first row and $n_1=n_0=0$. Then, counting the \fbox{2} boxes, we get that
\begin{eqnarray}\label{Ch4Eq1}
2k-j-\alpha_2 = \underbrace{k-\alpha_1}_{2^{nd} \text{ row}} + \underbrace{n_{\overline{1}}}_{3^{rd} \text{ row}} + \underbrace{\sum_{l=2}^{a+1} n_{\overline{\overline{2}}}}_{\text{other rows}}.
\end{eqnarray}
Since $\alpha$ is a partition of $k$ and it is defined by the shaded boxes, we have also that
\begin{eqnarray}\label{Ch4Eq2}
k= n_0+n_1+2n_{\overline{1}}+ (a+1)n_{a+1} + \sum_{l=2}^{a} l\cdot n_l + \sum_{l=2}^a l\cdot n_{\overline{l}}+ \sum_{l=2}^{a-1} (l+1)n_{\overline{\overline{l}}}.
\end{eqnarray}
Substituting \eqref{Ch4Eq2} in \eqref{Ch4Eq1}, we have that $j= \sum_{l=1}^{a+1} l\cdot n_l + \sum_{l=2}^a l\cdot n_{\overline{l}}+ \sum_{l=2}^{a-1} l\cdot n_{\overline{\overline{l}}}$.

\item For $\alpha_1 - \alpha_2 >0$, first we show that we cannot have $(\#1)_{R2}=\alpha_1-\alpha_2$. For this, we count the number of \fbox{l} boxes, for $l=1,\dots, a+1$, in the first row:
\begin{eqnarray*}
(\#1)_{R1} &=& 2k+j-\alpha_1 - (\alpha_1-\alpha_2) - \sum_{s=1}^a n_s, \\
(\#2)_{R1} &=& k-j + \alpha_1 - \alpha_2 - n_0 - n_{\overline{1}} - \sum_{s=2}^{a-1} n_{\overline{\overline{s}}} ,\\
(\#3)_{R1} &=& n_2 + n_{\overline{1}} \\
(\#l)_{R1} &=& n_{l-1} + n_{\overline{\overline{l-2}}} \hspace{0.7cm} \text{ for } l=4,\dots, a, \\
(\# a+1)_{R1} &=& n_a + n_{\overline{\overline{a-1}}}.
\end{eqnarray*}
Then, their sum is $3k-\alpha_1 +n_1$, since there are $3k-\alpha_1$ boxes in total in the first row, $n_1=0$ and there are no \fbox{$1$} boxes in the second row. 
Then, by the $\alpha$--condition, $(\#2)_{R1}= \alpha_1-\alpha_2 >0$, and we obtain that
\begin{eqnarray*}
\alpha_1-\alpha_2 = 2k-j-\alpha_2 - ((k-\alpha_1) +n_0) - n_{\overline{1}}- \sum_{l=2}^{a-1} n_{\overline{\overline{l}}},
\end{eqnarray*}
which simplifies as $k-j-n_0-n_{\overline{1}} - \sum_{l=2}^{a-1} n_{\overline{\overline{l}}} =0$.
By \eqref{Ch4Eq1}, we get that 
\begin{eqnarray*}
j= \sum_{l=1}^{a+1} l\cdot n_l + \sum_{l=2}^a l\cdot n_{\overline{l}}+ \sum_{l=2}^{a-1} l\cdot n_{\overline{\overline{l}}}.
\end{eqnarray*} 
\end{itemize}

It is immediate to check that for $b\neq a+1$, there is one tableau or none. 
\end{proof}

\begin{proof}[\textbf{Proof of Theorem \ref{ThmGF}:}]
The function $\mathcal{F}_{a,a}$ is the generating function of the coloured partitions with parts in $\mathcal{B}_a= \left\{ \overline{1},2,\overline{2},\dots, a, \overline{a}, \overline{a+1}\right\}$, for which both parts $i$ and $\overline{i}$ have weight $i$. 

To prove Theorem \ref{ThmGF} we stablish a bijection between coloured partitions with parts in $\mathcal{B}_a$, and corresponding Kronecker tableaux.

As before, the reduced Kronecker coefficients of the case $b=a$ in Family 1 can be written as Kronecker coefficients using Corollary \ref{StabilityBound}:
\begin{eqnarray*}
 \overline{g}^{(k)}_{(k^a), (k^a)}= g^{((a+2)k,k)}_{(3k,k^a), (3k,k^a)}.
\end{eqnarray*}
Then, using the result of  R. Orellana and C. Ballantine, Theorem \ref{ThmKT}, we obtain a combinatorial description of these coefficients in terms of Kronecker tableaux: $ \overline{g}^{(k)}_{(k^a), (k^a)}$ equals the number of Kronecker tableaux with type $(3k,k^a)/\alpha$ and shape $(3k,k^a)/\alpha$, where $\alpha$ is a partition of $k$ with $\ell(\alpha)\leq a+1$. 
We stablish a bijection between coloured partitions with parts in $\mathcal{B}_a$, and Kronecker tableaux  with shape $(3k,k^a)/\alpha$ and type $(3k,k^a)/\alpha$, where $\alpha$ is a partition of $k$  with $\ell(\alpha)\leq a+1$. 
The bijection is defined by the following algorithm. 
To a  coloured  partition  of $k$ with parts in $\mathcal{B}_a$, $\beta$,  we associate a Kronecker tableau $T(\beta)$ as follows.  

First, we identify each element of $\mathcal{B}_a$ to a column of height $a+1$:  
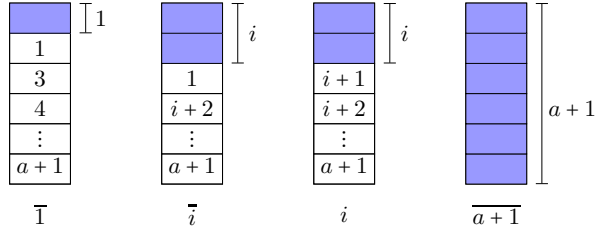
\begin{center}
\scalebox{0.8}{
\begin{tikzpicture}
\draw (2.5,0.5) rectangle (3.5,3.5); 
\draw (2.5,0.5) rectangle (3.5,1);
\draw (2.5,1) rectangle (3.5,1.5);
\draw (2.5,1.5) rectangle (3.5,2);
\draw (2.5,2) rectangle (3.5,2.5);
\draw (2.5,2.5) rectangle (3.5,3);
\draw[fill=blue!40] (2.5,3) rectangle (3.5,3.5);
\node at (3,0) {$\overline{1}$};
\node at (3,0.8) {$a+1$};
\node at (3,1.25) {$\vdots$};
\node at (3,1.75) {$4$};
\node at (3,2.25) {$3$};
\node at (3,2.75) {$1$};
\draw[|-|] (3.75,3) --(3.75,3.5);
\node at (4, 3.25) {$1$};

\draw (5,0.5) rectangle (6,3.5); 
\draw (5,0.5) rectangle (6,1);
\draw (5,1) rectangle (6,1.5);
\draw (5,1.5) rectangle (6,2);
\draw (5,2) rectangle (6,2.5);
\draw[fill=blue!40] (5,2.5) rectangle (6,3);
\draw[fill=blue!40] (5,3) rectangle (6,3.5);
\node at (5.5,0) {$\overline{i}$};
\node at (5.5,0.8) {$a+1$};
\node at (5.5,1.25) {$\vdots$};
\node at (5.5,1.75) {$i+2$};
\node at (5.5,2.25) {$1$};
\draw[|-|] (6.25,2.5) --(6.25,3.5);
\node at (6.5, 3) {$i$};

\draw (7.5,0.5) rectangle (8.5,3.5);  
\draw (7.5,0.5) rectangle  (8.5,1);
\draw (7.5,1) rectangle (8.5,1.5);
\draw (7.5,1.5) rectangle (8.5,2);
\draw (7.5,2) rectangle (8.5,2.5);
\draw[fill=blue!40] (7.5,2.5) rectangle (8.5,3);
\draw[fill=blue!40] (7.5,3) rectangle (8.5,3.5);
\node at (8,0) {$i$};
\node at (8,0.8) {$a+1$};
\node at (8,1.25) {$\vdots$};
\node at (8,1.75) {$i+2$};
\node at (8,2.25) {$i+1$};
\draw[|-|] (8.75,2.5) --(8.75,3.5);
\node at (9, 3) {$i$};

\draw[fill=blue!40] (10,0.5) rectangle (11,3.5); 
\draw[fill=blue!40] (10,0.5) rectangle (11,1);
\draw[fill=blue!40] (10,1) rectangle (11,1.5);
\draw[fill=blue!40] (10,1.5) rectangle (11,2);
\draw[fill=blue!40] (10,2) rectangle (11,2.5);
\draw[fill=blue!40] (10,2.5) rectangle (11,3);
\draw[fill=blue!40] (10,3) rectangle (11,3.5);
\node at (10.5,0) {$\overline{a+1}$};
\draw[|-|] (11.25,0.5) --(11.25,3.5);
\node at (11.75, 1.75) {$a+1$};
\end{tikzpicture}}
\end{center}
for $i \in \{2,3,\dots,a-1,a\}$.

Note that it is always possible to order the columns corresponding to the parts of $\beta$ in such a way that we obtain a \ssyt. 
We denote by $m_i$ the number of times that the part $i\in \mathcal{B}_a$ appears in $\beta$. Then, the column $i$ appears $m_i$ times in the \ssyt\ that we are building.
We read the partition  $\alpha$ from our \ssyt\ by counting the number of shaded boxes in each row: $\alpha_{a+1}= m_{\overline{a+1}}$, $\alpha_i = \alpha_{i+1} +m_i + m_{\overline{i}}$ for $i=2,\dots, a$, and $\alpha_1= \alpha_2 +m_{\overline{1}}$.

These columns correspond to the first columns on the left--hand side of $T(\beta)$. We build the rest of $T(\beta)$ of shape $(3k,k^a)/\alpha$ as follows: complete $i^{th}$ row with \fbox{$i$} boxes, for $i=2,\dots,a+1$, and complete first row with the remaining numbers of the type $(3k,k^a) / \alpha$ in weakly increasing order from left to right. 

For instance, the Kronecker tableau corresponding to $\alpha=(2,1)$ and $\lambda=\nu=(9,3,3,3)$ obtained by our algorithm taking $a=3$ and $\beta=(2,\overline{1})$, is 
\begin{center}
\scalebox{0.8}{
\begin{tikzpicture}
\draw (0,0.5) rectangle (1.5,2.5);
\draw (0,1) --(1.5,1);
\draw (0,1.5) --(1.5,1.5);
\draw (0,2) --(4.5,2);
\draw (0,2.5) --(4.5,2.5);
\draw (0.5,0.5) --(0.5,2.5);
\draw (1,0.5) --(1,2.5);
\draw (1.5,0.5) --(1.5,2.5);
\draw (2,2)--(2,2.5);
\draw (2.5,2)--(2.5,2.5);
\draw (3,2)--(3,2.5);
\draw (3.5,2)--(3.5,2.5);
\draw (4,2)--(4,2.5);
\draw (4.5,2)--(4.5,2.5);

\draw[fill=blue!40] (0,1.5) rectangle (0.5,2);
\draw[fill=blue!40] (0,2) rectangle (0.5,2.5);
\draw[fill=blue!40] (0.5,2) rectangle (1,2.5);

\node at (0.25,0.75) {4};
\node at (0.25,1.25) {3};
\node at (0.75,0.75) {4};
\node at (0.75,1.25) {3};
\node at (0.75,1.75) {1};
\node at (1.25,0.75) {4};
\node at (1.25,1.25) {3};
\node at (1.25,1.75) {2};
\node at (1.25,2.25) {1};
\node at (1.75,2.25) {1};
\node at (2.25,2.25) {1};
\node at (2.75,2.25) {1};
\node at (3.25,2.25) {1};
\node at (3.75,2.25) {2};
\node at (4.25,2.25) {3};
\end{tikzpicture}}
\end{center} 

To the check that the algorithm describes a bijection between the set of Kronecker tableaux and the set of coloured partitions we proceed as in the proof of Theorem \ref{ThmFam3}. Details can be found in \cite{Col16}.

\vspace{1cm}

For $b=a$ in Family 2, the reduced Kronecker coefficients are described as Kronecker coefficients as follows:
\begin{eqnarray*}
\overline{g}_{(k^a),\left((k+i)^a\right)}^{(k)} = g_{(n-ak,k^a), \left( n-a(k+i), (k+i)^a\right)}^{(n-k,k)},
\end{eqnarray*}
with $n\geq N$, for some $N$. 
We need to use different bounds for $N$, depending on whether the values of $k$ and $i$, in order to have that the sequences indexing the Kronecker coefficients are partitions. Otherwise, we could not apply Theorem \ref{ThmKT} to obtain a combinatorial interpretation.

For $k< \frac{(a+1)i}{2}$, we take $N=(a+3)k+(a+1)i$. Then, applying Theorem \ref{ThmKT}, we interpret them in terms of Kronecker tableaux: $\overline{g}^{(k)}_{\left( (k+i)^a \right) (k^b)}$ equals the number of Kronecker tableaux of shape $(3k+i,(k+i)^a)/\alpha$ and type $(3k+(a+1)i, k^a)/\alpha$, with $\alpha$ a partition of $k$ with $\ell(\alpha)\leq a+1$. These Kronecker tableaux have the following form
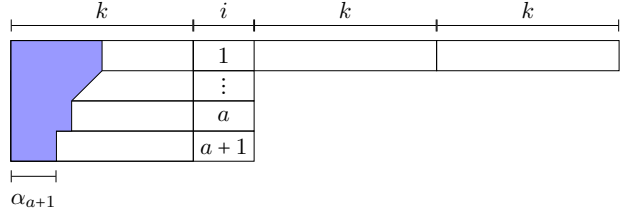
\begin{center}
\scalebox{0.8}{
\begin{tikzpicture}
\draw (0,0) rectangle (4,0.5);
\draw (0,0.5) rectangle (4,1);
\draw (0,1) rectangle (4,1.5);
\draw (3,0) --(3,2);
\draw (4,1.5) --(4,2);
\draw (3,1.5) rectangle (7,2);
\draw (7,1.5) rectangle (10,2);
\draw[fill=blue!40] (0,0) --(0.75,0) --(0.75,0.5) --(1,0.5) --(1,1) --(1.5,1.5) --(1.5,2) --(0,2) --(0,0);
\draw (0,0) rectangle (3,2);

\draw[|-|] (0,-0.25) --(0.75,-0.25);
\draw[|-|] (3,2.25) --(4,2.25);
\draw[|-|] (0,2.25) --(3,2.25);
\draw[|-|] (4,2.25) --(7,2.25);
\draw[|-|] (7,2.25) --(10,2.25);

\node at (0.375,-0.65) {$\alpha_{a+1}$};
\node at (3.5,2.5) {$i$};
\node at (1.5,2.5) {$k$};
\node at (5.5,2.5) {$k$};
\node at (8.5,2.5) {$k$};
\node at (3.5,0.25) {$a+1$};
\node at (3.5,0.75) {$a$};
\node at (3.5,1.25) {$\vdots$};
\node at (3.5,1.75) {$1$};
\end{tikzpicture}}
\end{center} 
As we observe in the \ssyt\ drawn above, there are always $k+i-\alpha_1$ columns of height $a+1$ that are filled with $j'$s in the $j^{th}$ row, for $j=1,\cdots,a+1$. We only have $(\alpha_1-\alpha_j-i)$ remaining $j'$s numbers to put on the \ssyt. If $i> \alpha_1-\alpha_j$, then we do not have enough numbers to fill those $k+i- \alpha_1$ columns of height $a+1$ and no shaded boxes. 

Suppose $i\leq \alpha_1-\alpha_j$, for all $j=1,\dots, a+1$. The condition over the reverse reading word implies that for $j\neq 1$, the \fbox{$j$} box only appears in the first and in the $j^{th}$ row. This implies that in the $j^{th}$ row there are only \fbox{1} and \fbox{$j$} boxes. For the \fbox{1} boxes, there are at least $k-\alpha_j-(\alpha_1-i-\alpha_j)$ of them in the $j^{th}$ row. Then, for $j=1,\dots, a+1$, $(\# 1)_{Rj} \geq k+i-\alpha_1$.

In order to avoid the columns 
\scalebox{0.7}{
\begin{tikzpicture}
\draw (0,0) rectangle (0.5,0.5);
\draw (0,0.5) rectangle (0.5,1);
\node at (0.25,0.25) {1};
\node at (0.25,0.75) {1};
\end{tikzpicture}}, $\alpha_j-\alpha_{j+1} \geq k+i-\alpha_1$. This implies that $\alpha_j - \alpha_{j+1} \geq 1$, and that $k=|\alpha| \geq a\cdot i$, which is a contradiction because we are in the case $k< \frac{(a+1)i}{2}$.
We conclude that for $k< \frac{(a+1)i}{2}$, the reduced Kronecker coefficient is zero. This explains the first zeros of the sequence of reduced Kronecker coefficients, once we fix $i$. The rest of the initial zeros are explained in the other case. 

For $k\geq \frac{(a+1)i}{2}$, we take $N=(a+3)k$. Then, applying Theorem \ref{ThmKT}, we interpret them in terms of Kronecker tableaux: $\overline{g}^{(k)}_{\left( (k+i)^a \right) (k^b)}$ equals the number of Kronecker tableaux of shape $(3k,k^a)/\alpha$ and type $(3k-ai, (k+i)^a)/\alpha$, with $\alpha$ a partition of $k$ with $\ell(\alpha)\leq a+1$. Let us see a general idea about how are these Kronecker tableaux.
First, we observe that for $j\in \{2,\dots, a+1\}$ there should be $k+i-\alpha_j$ \fbox{j} boxes in total and in the $j^{th}$th row there are only $k-\alpha_j$ boxes. Using the condition of the reverse reading word, we realize that for all $j\in \{1,\dots, a\}$, $\alpha_{j} - \alpha_{j+1} \geq i$. Then, we can fill $i$ boxes in the first row with \fbox{$j$}. In particular, this implies that $a \leq \ell(\alpha) \leq a+1$. 
Setting these conditions over all partitions, for $k < \frac{a(a+1)i}{2}=d\cdot i$ there is no valid partition. The first valid positive integer is $k=d\cdot i$ with the partition $\alpha=\left(ai, (a-1)i,\dots, 2a, a\right)$. That is why we have the rest of initial zeros. 

Let $k\geq d\cdot i$. First, consider the translation $k \longmapsto k + d\cdot i$, that makes the initial zeros disappear. Therefore, we consider the reduced Kronecker coefficients
\begin{eqnarray*}
\left\{ \overline{g}^{(k+d\cdot i)}_{\left( (k + (d+1)i)^a \right), \left((k+d\cdot i)^a\right)} \right\}_{k,i\geq 0}.
\end{eqnarray*}
One more time, we express them in terms of Kronecker coefficients using the Murnaghan's theorem and taking $N=(a+3)d + (a+1)di$. Applying Theorem \ref{ThmKT}, our reduced Kronecker coefficients count the number of Kronecker tableaux of shape $\lambda/\alpha= \left(3k+d\cdot i, (k+d\cdot i)^a \right)/\alpha$, type $\nu/\alpha=\left( 3k+(d-a)i, (k+(d+1)i)^a \right)/\alpha$ and associated partition $\alpha \vdash k+d\cdot i$, with $\ell(\alpha) \leq a+1$. 

To prove that their generating function is $\mathcal{F}_{a,a}$ we stablish a bijection between coloured partitions of $k$ with parts in $\mathcal{B}_a$, and Kronecker tableaux of shape $\lambda/\alpha$, type $\nu/\alpha$ and associated partition $\alpha \vdash k+d\cdot i$, with $\ell(\alpha) \leq a+1$. 

The bijection is defined by the following algorithm.
To a coloured partition of $k$ with parts in $\mathcal{B}_a$, we associate a Kronecker tableau $T(\beta)$ as follows. 
First, we identify each element of $\mathcal{B}_a$ to a column of height $a+1$. It is the same identification than for Family 1: 
\begin{center}
\scalebox{0.8}{
\begin{tikzpicture}
\draw (2.5,0.5) rectangle (3.5,3.5); 
\draw (2.5,0.5) rectangle (3.5,1);
\draw (2.5,1) rectangle (3.5,1.5);
\draw (2.5,1.5) rectangle (3.5,2);
\draw (2.5,2) rectangle (3.5,2.5);
\draw (2.5,2.5) rectangle (3.5,3);
\draw[fill=blue!40] (2.5,3) rectangle (3.5,3.5);
\node at (3,0) {$\overline{1}$};
\node at (3,0.8) {$a+1$};
\node at (3,1.25) {$\vdots$};
\node at (3,1.75) {$4$};
\node at (3,2.25) {$3$};
\node at (3,2.75) {$1$};
\draw[|-|] (3.75,3) --(3.75,3.5);
\node at (4, 3.25) {$1$};

\draw (5,0.5) rectangle (6,3.5); 
\draw (5,0.5) rectangle (6,1);
\draw (5,1) rectangle (6,1.5);
\draw (5,1.5) rectangle (6,2);
\draw (5,2) rectangle (6,2.5);
\draw[fill=blue!40] (5,2.5) rectangle (6,3);
\draw[fill=blue!40] (5,3) rectangle (6,3.5);
\node at (5.5,0) {$\overline{j}$};
\node at (5.5,0.8) {$a+1$};
\node at (5.5,1.25) {$\vdots$};
\node at (5.5,1.75) {$j+2$};
\node at (5.5,2.25) {$1$};
\draw[|-|] (6.25,2.5) --(6.25,3.5);
\node at (6.5, 3) {$j$};

\draw (7.5,0.5) rectangle (8.5,3.5);  
\draw (7.5,0.5) rectangle  (8.5,1);
\draw (7.5,1) rectangle (8.5,1.5);
\draw (7.5,1.5) rectangle (8.5,2);
\draw (7.5,2) rectangle (8.5,2.5);
\draw[fill=blue!40] (7.5,2.5) rectangle (8.5,3);
\draw[fill=blue!40] (7.5,3) rectangle (8.5,3.5);
\node at (8,0) {$j$};
\node at (8,0.8) {$a+1$};
\node at (8,1.25) {$\vdots$};
\node at (8,1.75) {$j+2$};
\node at (8,2.25) {$j+1$};
\draw[|-|] (8.75,2.5) --(8.75,3.5);
\node at (9, 3) {$j$};

\draw[fill=blue!40] (10,0.5) rectangle (11,3.5); 
\draw[fill=blue!40] (10,0.5) rectangle (11,1);
\draw[fill=blue!40] (10,1) rectangle (11,1.5);
\draw[fill=blue!40] (10,1.5) rectangle (11,2);
\draw[fill=blue!40] (10,2) rectangle (11,2.5);
\draw[fill=blue!40] (10,2.5) rectangle (11,3);
\draw[fill=blue!40] (10,3) rectangle (11,3.5);
\node at (10.5,0) {$\overline{a+1}$};
\draw[|-|] (11.25,0.5) --(11.25,3.5);
\node at (11.75, 1.75) {$a+1$};
\end{tikzpicture}}
\end{center}
for $j\in \{2,3,\dots,a-1,a\}$.
If we write $\beta$ as $\left(\overline{1}^{m_{\overline{1}}}\overline{2}^{m_{\overline{2}}} 2^{m_2} \dots \overline{a+1}^{m_{\overline{a+1}}}\right)$, then $m_i$ will denote the number of times that the column $i$ appears in the \ssyt\ that we are building. 

We continue adding the following columns, $i$ times each one: 
\begin{center}
\scalebox{0.8}{
\begin{tikzpicture}
\draw (0,0) rectangle (1,3.5); 
\draw (0,0) rectangle (1,0.5); 
\draw (0,0.5) rectangle (1,1);
\draw (0,1) rectangle (1,1.5);
\draw (0,1.5) rectangle (1,2);
\draw (0,2) rectangle (1,2.5);
\draw (0,2.5) rectangle (1,3);
\draw[fill=blue!40] (0,3) rectangle (1,3.5);
\node at (0.5,0.25) {$a+1$};
\node at (0.5,0.75) {$\vdots$};
\node at (0.5,1.25) {$5$};
\node at (0.5,1.75) {$4$};
\node at (0.5,2.25) {$3$};
\node at (0.5,2.75) {$2$};
\draw[|-|] (1.25,0) --(1.25,3);
\node at (1.5, 1.5) {$a$};
\draw[|-|] (1.25,3) --(1.25,3.5);
\node at (1.5, 3.25) {$1$};

\node at (2.25, 1.75) {$\dots$};

\draw (3,0) rectangle (4,3.5); 
\draw (3,0.5) --(4,0.5);
\draw (3,1) rectangle (4,1.5);
\draw (3,1.5) rectangle (4,2);
\draw (3,2) rectangle (4,2.5);
\draw[fill=blue!40] (3,2.5) rectangle (4,3);
\draw[fill=blue!40] (3,3) rectangle (4,3.5);
\node at (3.5,0.25) {$a+1$};
\node at (3.5,0.75) {$\vdots$};
\node at (3.5,1.25) {$j+2$};
\node at (3.5,1.75) {$j+1$};
\node at (3.5,2.25) {$j$};
\draw[|-|] (4.25,0) --(4.25,2.5);
\node at (5, 1.5) {$a+1-j$};
\draw[|-|] (4.25,2.5) --(4.25,3.5);
\node at (4.75, 3) {$j-1$};

\node at (6.5, 1.75) {$\dots$};

\draw (7.25,0) rectangle (8.25,3); 
\node at (7.75,0.25) {$a+1$};
\draw (7.25,0) rectangle (8.25,0.5); 
\draw[fill=blue!40] (7.25,0.5) rectangle (8.25,1);
\draw[fill=blue!40] (7.25,1) rectangle (8.25,1.5);
\draw[fill=blue!40] (7.25,1.5) rectangle (8.25,2);
\draw[fill=blue!40] (7.25,2) rectangle (8.25,2.5);
\draw[fill=blue!40] (7.25,2.5) rectangle (8.25,3);
\draw[fill=blue!40] (7.25,3) rectangle (8.25,3.5);
\draw[|-|] (8.5,0.5) --(8.5,3.5);
\node at (8.75, 2) {$a$};
\end{tikzpicture}}
\end{center}
for $j \in \{3,\dots,a\}$.
Note that it is always possible to order the columns corresponding to the parts of $\beta$ and these last extra columns in such a way that we obtain a \ssyt. 

We read the partition $\alpha$ from our \ssyt  by counting the number of shaded boxes in each row: $\alpha_{a+1} = m_{\overline{a+1}}$, $\alpha_{j} = \alpha_{j+1} + m_j + m_{\overline{j}} + i$, for $j=2,\dots,a$, and $\alpha_1 = \alpha_2 + m_{\overline{1}} + i$. 
These columns are the left columns of $T(\beta)$. We build the rest of $T(\beta)$ as follows: complete $j^{th}$ row with \fbox{$j$} boxes, for $j=2,\dots, a+1$, and complete the first row with the remaining numbers of the already known type of the tableaux, $\left( 3k+(3n-a)i, (k+(d+1)i)^a \right)/\alpha$, in weakly increasing order from left to right. 

For instance, take $a=3$, $k=6$ and $i=1$ and consider $\beta = (\overline{4}, \overline{2}) \vdash 6$. Applying the algorithm defined above, we obtain the following Kronecker tableau of shape $(36,12,12,12)/\alpha$ and type $(33,13,13,13)/\alpha$, with $\alpha=(5,4,2,1) \vdash 12$:
\begin{center}
\scalebox{0.65}{
\begin{tikzpicture}
\draw[fill=blue!40] (0,0) --(0.5,0) --(0.5,0.5) --(1,0.5) --(1,1) --(2,1) --(2,1.5) --(2.5,1.5) --(2.5,2) --(0,2) --(0,0);
\draw (0,0) --(6,0);
\draw (0,0.5) --(6,0.5);
\draw (0,1) --(6,1);
\draw (0,1.5) --(18,1.5);
\draw (0,2) --(18,2);
\draw (0.5,0) --(0.5,2);
\draw (1,0) --(1,2);
\draw (1.5,0) --(1.5,2);
\draw (2,0) --(2,2);
\draw (2.5,0) --(2.5,2);
\draw (3,0) --(3,2);
\draw (3.5,0) --(3.5,2);
\draw (4,0) --(4,2);
\draw (4.5,0) --(4.5,2);
\draw (5,0) --(5,2);
\draw (5.5,0) --(5.5,2);
\draw (6,0) --(6,2);
\draw (6.5,1.5) --(6.5,2);
\draw (7,1.5) --(7,2);
\draw (7.5,1.5) --(7.5,2);
\draw (8,1.5) --(8,2);
\draw (8.5,1.5) --(8.5,2);
\draw (9,1.5) --(9,2);
\draw (9.5,1.5) --(9.5,2);
\draw (10,1.5) --(10,2);
\draw (10.5,1.5) --(10.5,2);
\draw (11,1.5) --(11,2);
\draw (11.5,1.5) --(11.5,2);
\draw (12,1.5) --(12,2);
\draw (12.5,1.5) --(12.5,2);
\draw (13,1.5) --(13,2);
\draw (13.5,1.5) --(13.5,2);
\draw (14,1.5) --(14,2);
\draw (14.5,1.5) --(14.5,2);
\draw (15,1.5) --(15,2);
\draw (15.5,1.5) --(15.5,2);
\draw (16,1.5) --(16,2);
\draw (16.5,1.5) --(16.5,2);
\draw (17,1.5) --(17,2);
\draw (17.5,1.5) --(17.5,2);
\draw (18,1.5) --(18,2);

\node at (0.75,0.25) {4};
\node at (1.25,0.25) {4};
\node at (1.75,0.25) {4};
\node at (2.25,0.25) {4};
\node at (2.75,0.25) {4};
\node at (3.25,0.25) {4};
\node at (3.75,0.25) {4};
\node at (4.25,0.25) {4};
\node at (4.75,0.25) {4};
\node at (5.25,0.25) {4};
\node at (5.75,0.25) {4};
\node at (17.75,1.75) {4};

\node at (1.25,0.75) {1};
\node at (1.75,0.75) {3};
\node at (2.25,0.75) {3};
\node at (2.75,0.75) {3};
\node at (3.25,0.75) {3};
\node at (3.75,0.75) {3};
\node at (4.25,0.75) {3};
\node at (4.75,0.75) {3};
\node at (5.25,0.75) {3};
\node at (5.75,0.75) {3};
\node at (17.25,1.75) {3};

\node at (2.25,1.25) {2};
\node at (2.75,1.25) {2};
\node at (3.25,1.25) {2};
\node at (3.75,1.25) {2};
\node at (4.25,1.25) {2};
\node at (4.75,1.25) {2};
\node at (5.25,1.25) {2};
\node at (5.75,1.25) {2};
\node at (16.75,1.75) {2};
\node at (16.25,1.75) {1};

\node at (2.75,1.75) {1};
\node at (3.25,1.75) {1};
\node at (3.75,1.75) {1};
\node at (4.25,1.75) {1};
\node at (4.75,1.75) {1};
\node at (5.25,1.75) {1};
\node at (5.75,1.75) {1};
\node at (6.25,1.75) {1};
\node at (6.75,1.75) {1};
\node at (7.25,1.75) {1};
\node at (7.75,1.75) {1};
\node at (8.25,1.75) {1};
\node at (8.75,1.75) {1};
\node at (9.25,1.75) {1};
\node at (9.75,1.75) {1};
\node at (10.25,1.75) {1};
\node at (10.75,1.75) {1};
\node at (11.25,1.75) {1};
\node at (11.75,1.75) {1};
\node at (12.25,1.75) {1};
\node at (12.75,1.75) {1};
\node at (13.25,1.75) {1};
\node at (13.75,1.75) {1};
\node at (14.25,1.75) {1};
\node at (14.75,1.75) {1};
\node at (15.25,1.75) {1};
\node at (15.75,1.75) {1};
\end{tikzpicture}}
\end{center}
The map defined by this algorithm is a bijection, proceeding as in the proof of Theorem \ref{ThmFam3}. Complete details can be found in \cite{Col16}.

\end{proof}

\section{Plane Partitions and Reduced Kronecker Coefficients}\label{PP}

In this section we establish a link between our families of reduced Kronecker coefficients and  plane partitions. 

 A \emph{plane partition} is a finite subset $\mathcal{P}$ of positive integer lattice points, $\{(i,j,k)\} \subset \mathbb{N}^3$, such that if $(r,s,t)$ lies in $\mathcal{P}$ and if $(i,j,k)$ satisfies $1\leq i \leq r$, $1\leq j \leq s$ and $1\leq k\leq t$,  then $(i,j,k)$ also lies in $\mathcal{P}$. 
 Let $\mathcal{B}(r,s,t)$ be the set of plane partitions fitting in a $r\times s$ rectangle, and with biggest part equal to or less than $t$. For instance, the following is a plane partition of 28 in $\mathcal{B}_{4,3,5}$
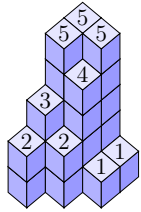
\begin{center}  
  {\scalebox{0.8}{
    \begin{tikzpicture}[x=(220:0.4cm), y=(-40:0.4cm), z=(90:0.4242cm)]
\foreach \m [count=\y] in {{5,5,3,2},{5,4,2},{1,1}}{
  \foreach \n [count=\x] in \m {
  \ifnum \n>0
      \foreach \z in {1,...,\n}{
        \draw [fill=blue!30] (\x+1,\y,\z) -- (\x+1,\y+1,\z) -- (\x+1, \y+1, \z-1) -- (\x+1, \y, \z-1) -- cycle;
        \draw [fill=blue!40] (\x,\y+1,\z) -- (\x+1,\y+1,\z) -- (\x+1, \y+1, \z-1) -- (\x, \y+1, \z-1) -- cycle;
        \draw [fill=blue!10] (\x,\y,\z)   -- (\x+1,\y,\z)   -- (\x+1, \y+1, \z)   -- (\x, \y+1, \z) -- cycle;  
      }
      \node at (\x+0.5, \y+0.5, \n) {\n};
 \fi
 }
}
\end{tikzpicture}}}
\end{center}

MacMahon obtained in 1920 the generating function for the plane partitions fitting in a rectangle. 
\begin{theorem}[P. MacMahon, \cite{MR2417935}]\label{ThmPPIni}
 The generating function for plane partitions in $\mathcal{B}(r,s,t)$ is given by
$$
  pp_t(x;r,s)= \prod_{i=1}^r \prod_{j=1}^s \frac{1-x^{i+j+t-1}}{1-x^{i+j-1}}.
$$
\end{theorem}
We can reformulate MacMahon's result in the following way.
 \begin{corollary}\label{ThmPPfg}
Let $l=\min(r,s)$ and $n=\max(r,s)$. Then, the generating function for the plane partitions fitting inside an $r\times s$ rectangle is
$$
 \prod_{j=l}^{n} \left( \frac{1}{1-x^j} \right)^{l} \cdot \prod_{i=1}^{l-1}\left(\frac{1}{1-x^i}\right)^i \left( \frac{1}{1-x^{n+ i}}\right)^{l-i}.
$$
  \end{corollary}
Note that in this case, the coefficient of $x^n$ corresponds to the plane partition fitting inside a $r\times s$ rectangle with biggest part less than or equal to $n$.
\begin{theorem}\label{ThmPP}
We have the following combinatorial interpretation for our families of reduced Kronecker coefficients.
\begin{enumerate}
\item The reduced Kronecker coefficient of Family 1, $\overline{g}_{(k^a),(k^a)}^{(k)}$, counts the number of plane partitions of $k$ fitting inside a $2\times a$ rectangle. 

\item The reduced Kronecker coefficient of Family 2 after shifting the initial zeros are described as:
\begin{eqnarray*}
\left\{ \overline{g}^{(k+d\cdot i)}_{\left( (k + (d+1)i)^a \right), \left((k+d\cdot i)^a\right)} \right\}_{k,i\geq 0},
\end{eqnarray*}
with $d=\frac{a(a+1)}{2}$. Then, for any $i$ fixed, they count the number of plane partitions of $k$ fitting inside a $2\times a$ rectangle. 

\item We have the following combinatorial interpretation of the reduced Kronecker coefficients $\overline{\overline{g}}_a(j)$, which are the stable values of the diagonals in Family 3:
\begin{eqnarray*}
\overline{\overline{g}}_a(j)  = \sum_{l=0}^j \# \left\{ \begin{array}{c}  \text{plane partitions of } l \\ \text{in } 3\times (a-1) \text{ rectangle}\end{array} \right\} \# \left\{ \begin{array}{c}  \text{plane partitions of }j- l \\ \text{in } 2\times 1  \text{ rectangle}\end{array} \right\} .
\end{eqnarray*}
\end{enumerate}
\end{theorem}
\begin{proof}
The descriptions corresponding to Families 1 and 2 are proved comparing directly the generating function of both combinatorial objects. 
For Family 3, consider the generating function of the plane partitions fitting inside a $3\times (a-1)$ rectangle. By Corollary \ref{ThmPPfg}, it can be written as
 \begin{eqnarray*}
 \mathcal{H}_a = \frac{1}{(1-x)(1-x^2)^2 (1-x^3)^3 \dots (1-x^{a-1})^3(1-x^a)^2(1-x^{a+1})}.
 \end{eqnarray*}
 By Theorem \ref{ThmFam3}, the generating function of the reduced Kronecker coefficients $\overline{\overline{g}}_a(j)$ is $\mathcal{G}_a$, which is related with $\mathcal{H}_a$ by
 \begin{eqnarray}\label{HGEq}
  \mathcal{H}_a = (1-x)(1-x^2) \mathcal{G}_a.
 \end{eqnarray}
 Then, we express the coefficients appearing in the expansion of $\mathcal{G}_a$ in terms of the coefficients of $\mathcal{H}_a$. Let $\mathcal{G}_a = \sum_n q_n x^n$ be the expansion for $\mathcal{G}_a$ and $\mathcal{H}_a = \sum_n r_n x^n$ the corresponding for $\mathcal{H}_a$. Then, $r_n$ is the number of plane partitions fitting inside a $3\times (a-1)$ rectangle. 
 
 Expanding both sides of $\eqref{HGEq}$ and equating the coefficients, we get the following recursive relation
 \begin{eqnarray*}
\begin{array}{l}
r_0 = q_0, \\
r_1 = q_1- q_0, \\
r_2 = q_2 - q_1 -q_0, \\
r_n = q_n - q_{n-1} - q_{n-2} + q_{n-3} \hspace{1cm} \text{ for all } n\geq 3.
\end{array}
\end{eqnarray*}
The coefficients $q_n$ can be expressed in terms of the $r_n$ coefficients.
\begin{lemma}\label{Lemma2}
 With the same notation as above, 
 \begin{eqnarray*}
q_n = \sum_{m=0}^n \left( \left\lfloor \frac{n-m}{2}\right\rfloor + 1 \right) r_m.
 \end{eqnarray*}
\end{lemma}
Finally, observe that the coefficients that appears in Lemma ($\ref{Lemma2}$) count the number of plane partitions fitting inside a $2\times 1$ rectangle. 
 \end{proof}
 
  \subsection{Saturation Hypothesis}
 \mbox{}\vspace{0.25cm}\\
Let denote by $\{ C(\alpha^1,\dots, \alpha^n) \}$ any family of coefficients depending on the partitions $\alpha^1,\dots, \alpha^n$. The family  $\{ C(\alpha^1,\dots, \alpha^n) \}$ satisfies the \emph{saturation hypothesis} if the conditions 
$ C(\alpha^1,\dots, \alpha^n) > 0 $ and $ C(s\cdot \alpha^1, \dots, s\cdot \alpha^n) > 0 $ for all $s>1$ are equivalent, where $s\cdot \alpha = (s\cdot \alpha_1, s\cdot \alpha_2,\dots)$.  

The  Littlewood--Richardson coefficients  satisfy the saturation hypothesis, as was shown by Knutson and Tao in \cite{MR1671451}. On the other hand, the Kronecker coefficients are known not to satisfy it. For instance, consider the following family of Kronecker coefficients given in \cite{MR2570451}: $\left\{g_{(n,n),(n,n)}^{(n,n)}\right\}_{n\geq 0}$. Then, the coefficient is 1 when $n$ is even, and 0 otherwise. This family shows that Kronecker coefficients do not satisfy the saturation hypothesis in general. 

In \cite{Klya04} and \cite{MR2105706}, Kirillov and Klyachko  conjecture that the  reduced Kronecker coefficients satisfy the saturation hypothesis.
 From the combinatorial interpretation of our familiyes of reduced Kronecker coefficients in terms of plane partitions given in Theorem \ref{ThmPP}, we verify that for these coefficients the saturation hypothesis holds.
 \begin{proposition}
The saturation hypothesis holds for the coefficients of Family 1, $\overline{g}_{(k^a), (k^a)}^{(k)}$. In fact, $\overline{g}_{((sk)^a), ((sk)^a)}^{(sk)} >0$ for all $s\geq 1$. Moreover, the sequences of coefficients obtained by, either fixing $k$ or $a$, and letting the other parameter grow are weakly increasing. 
\end{proposition}

\begin{proof}
By Theorem \ref{ThmGF}, $\overline{g}_{((sk)^a), ((sk)^a)}^{(sk)} $ counts the number of plane partitions of $sk$ fitting in a $2\times a$ rectangle. 
Consider the plane partition with only one part, \fbox{sk}. There always exists this plane partition, whether it is $s\geq 1$.   
Then,  $\overline{g}_{((sk)^a), ((sk)^a)}^{(sk)} >0$ for any $s\geq 1$.

To prove that the coefficients are weakly increasing, we  show that any plane partition of $k$ fitting in a $2\times a$ rectangle can be identify with a plane partition of $k+1$ fitting in the same rectangle. 
 Consider a plane partition of $k$ fitting in a $2\times a$ rectangle. We identify this plane partition with the plane partition of $k+1$ fitting in the same rectangle. The identification is as follows:
 \begin{center}
 \scalebox{0.7}{
  \begin{tikzpicture}
   \draw (0,0) rectangle (1,0.5);
   \draw (1,0) rectangle (2,0.5);
   \draw (2,0) rectangle (3.5,0.5);
   \draw (3.5,0) rectangle (4.5,0.5);
   
   \draw (0,0.5) rectangle (1,1);
   \draw (1,0.5) rectangle (2,1);
   \draw (2,0.5) rectangle (3.5,1);
   \draw (3.5,0.5) rectangle (4.5,1);
   
   \draw[|-|] (0,1.25) --(4.5,1.25);
   
   \node at (0.5,0.25) {$\alpha_{21}$};
   \node at (1.5,0.25) {$\alpha_{22}$};
   \node at (2.75,0.25) {$\cdots$};
   \node at (4,0.25) {$\alpha_{2,a}$};

   \node at (0.5,0.75) {$\alpha_{11}$};
   \node at (1.5,0.75) {$\alpha_{12}$};
   \node at (2.75,0.75) {$\cdots$};
   \node at (4,0.75) {$\alpha_{1,a}$};
   \node at (2.25,1.5) {$a$};
   
   \draw[|->] (5,0.5) -- (6.5,0.5);
   
   \draw (7,0) rectangle (8.5,0.5);
   \draw (8.5,0) rectangle (9.5,0.5);
   \draw (9.5,0) rectangle (11,0.5);
   \draw (11,0) rectangle (12,0.5);
   
   \draw (7,0.5) rectangle (8.5,1);
   \draw (8.5,0.5) rectangle (9.5,1);
   \draw (9.5,0.5) rectangle (11,1);
   \draw (11,0.5) rectangle (12,1);
   
   \draw[|-|] (7,1.25) --(12,1.25);
   
   \node at (7.75,0.25) {$\alpha_{21}$};
   \node at (9,0.25) {$\alpha_{22}$};
   \node at (10.25,0.25) {$\cdots$};
   \node at (11.5,0.25) {$\alpha_{2,a}$};

   \node at (7.75,0.75) {$\alpha_{11}+1$};
   \node at (9,0.75) {$\alpha_{12}$};
   \node at (10.25,0.75) {$\cdots$};
   \node at (11.5,0.75) {$\alpha_{1,a}$};
   \node at (9.5,1.5) {$a$};
  \end{tikzpicture}}
 \end{center}

Then, the set of plane partitions of $k+1$ fitting in a $2\times a$ rectangle has at least as many elements as the set of plane partitions of $k$ fitting in the same rectangle. 
\end{proof}

For Family 2, we have to consider the reduced Kronecker coefficients after the initial zeros. Recall that $d=\frac{a(a+1)}{2}$.
 \begin{corollary}
The saturation hypothesis holds for the coefficients 
\begin{eqnarray*}
\overline{g}_{((k+(d+1)i)^a), ((k+d\cdot i)^a)}^{(k+d\cdot i)}.
\end{eqnarray*}
Moreover, the sequences of coefficients obtained by fixing $a$ and letting $k$ grow are weakly increasing. 
\end{corollary}
Observe that in this case, once we fix $k$, the sequence obtained letting $a$ grows is not weakly increasing because the number of initial zeros depends on $a$. 

Finally, for Family 3, we have the corresponding result concerning the stable values of the diagonals, $\overline{\overline{g}}_a(j)$.
\begin{corollary}
The saturation hypothesis holds for the coefficients $\overline{\overline{g}}_a(j)$. In fact, $\overline{\overline{g}}_a(sj) >0$ for all $s\geq 1$, where $\overline{\overline{g}}_a(sj)$ denotes the associated reduced Kronecker coefficient with its three partitions multiplied by $s$. Moreover, the sequences of coefficients obtained by, either fixing $i$ or $a$, and then letting $k$ grow, are weakly increasing. 
\end{corollary}

\begin{proof}
By Theorem \ref{ThmFam3}, $\overline{\overline{g}}_a(j)$ has the following combinatorial description in terms of plane partitions:
\begin{eqnarray*}
\overline{\overline{g}}_a(j) = \sum_{l=0}^j \# \left\{ \begin{array}{c}  \text{plane partitions of } l \\ \text{in } 3\times (a-1) \text{ rectangle} \end{array} \right\} \# \left\{ \begin{array}{c}  \text{plane partitions of }j- l \\ \text{in } 2\times 1 \text{ rectangle} \end{array} \right\} .
\end{eqnarray*}
There always exists a pair of plane partitions with a plane partition of $l\in \{0,\dots, sj\}$, fitting in a $3\times (a-1)$ rectangle, and a plane partition of $sj-l$ fitting in a $2\times 1$ rectangle: for $s\geq 1$, we consider the following pair 
\scalebox{0.7}{
\begin{tikzpicture}
 \draw (0,0) rectangle (1,0.5);
 \draw (1.25,0) rectangle (3.25, 0.5);
 
 \node at (0.5,0.25) {$sl$};
 \node at (2.25, 0.25) {$s(j-l)$};
\end{tikzpicture}},
with $l\in \{0,\dots, j\}$.

Indeed, any pair of plane partitions associated to $j$ defines a pair of plane partitions associated to $j+1$. 
Let us consider a pair of plane partitions, where the first partition is any plane partition of $l \in \{0,\dots, j\}$ fitting in a $3\times (a-1)$ rectangle and the second is any plane partition of $j-l$ fitting in a $2\times 1$ 
  \begin{center}
  \scalebox{0.8}{
  \begin{tikzpicture}
   \draw (0,0) rectangle (1,0.5);
   \draw (1,0) rectangle (2,0.5);
   \draw (2,0) rectangle (4,0.5);
   \draw (4,0) rectangle (5.5,0.5);
   
   \draw (0,0.5) rectangle (1,1);
   \draw (1,0.5) rectangle (2,1);
   \draw (2,0.5) rectangle (4,1);
   \draw (4,0.5) rectangle (5.5,1);
   
   \draw (0,1) rectangle (1,1.5);
   \draw (1,1) rectangle (2,1.5);
   \draw (2,1) rectangle (4,1.5);
   \draw (4,1) rectangle (5.5,1.5);
   
   \draw (6,0) rectangle (7,0.75);
   \draw (6,0.75) rectangle (7,1.5);

   \draw[|-|] (0,1.75) --(5.5,1.75);
   \draw[|-|] (6,1.75) --(7,1.75);
   
   \node at (0.5,0.25) {$\alpha_{31}$};
   \node at (1.5,0.25) {$\alpha_{32}$};
   \node at (3.1,0.25) {$\cdots$};
   \node at (4.75,0.25) {$\alpha_{3,a-1}$};

   \node at (0.5,0.75) {$\alpha_{21}$};
   \node at (1.5,0.75) {$\alpha_{22}$};
   \node at (3.1,0.75) {$\cdots$};
   \node at (4.75,0.75) {$\alpha_{2,a-1}$};
   
   \node at (0.5,1.25) {$\alpha_{11}$};
   \node at (1.5,1.25) {$\alpha_{12}$};
   \node at (3.1,1.25) {$\cdots$};
   \node at (4.75,1.25) {$\alpha_{1,a-1}$};
   
   \node at (6.5,0.375) {$\beta_2$};
   \node at (6.5,1.125) {$\beta_1$};
   
   \node at (3.1,2) {$a-1$};  
   \node at (6.5,2) {$1$}; 
  \end{tikzpicture}}
 \end{center}
 If we add one to the first element of the first plane partition, $\alpha_{11}+1$, we obtain a plane partition of $l^\prime =l+1 \in [0,\dots, j+1]$ fitting in a $3\times (a-1)$ rectangle. 
 The other plane partition can be seen as a plane partition of $j+1-l^\prime=j-l$ that fits in a $2\times 1$ rectangle. 
\end{proof}

\section{Quasipolynomiality} 

In Theorem \ref{ThmGF} and Theorem \ref{ThmFam3} we compute the generating functions $\mathcal{F}_a$ and $\mathcal{G}_a$ for the reduced Kronecker coefficients associated to Families 1 and 3. These are the resulting implications of this calculation.
  \begin{theorem}\label{ThmQuasiPoly}
  Let $\mathcal{F}_a$ be the generating function for the reduced Kronecker coefficients $\overline{g}_{(k^a),(k^a)}^{(k)}$.
  Let $\mathcal{G}_{a}$ be the generating function for  the reduced Kronecker coefficients $\overline{\overline{g}}_a(j)$.

Let $\ell$ be the least common multiple of $1, 2, \ldots, a, a+1$.
\begin{enumerate}
\item The generating function $\mathcal{F}_{a}$ can be rewritten as $\mathcal{F}_{a}= \displaystyle{\frac{P_a(x)}{(1-x^{\ell})^{2a}}}$, where $P_a(x)$ is a product of cyclotomic polynomials. 

Moreover, $\deg(P_a(x)) =2a\ell - (a+2)a< 2a\ell-1$. 

\item The generating function $\mathcal{G}_{a}$ can be rewritten as $\displaystyle{\mathcal{G}_{a}= \frac{Q_a(x)}{(1-x^{\ell})^{3a-1}}}$, where $Q_a(x)$ is a product of cyclotomic polynomials.

 Moreover $\deg(Q_a(x)) =\ell (3a-1) - \frac{3}{2}(a^2+a)<\ell(3a-1)-1$. 

\item  The  coefficients $\overline{g}^{(k)}_{(k^a), (k^a)}$ are described by a quasipolynomial of degree $2a-1$ and period dividing $\ell$. In fact, we have checked that the period is exactly $l$ for $a$ less than 10.

\item  The  coefficients $\overline{\overline{g}}_a(j)$ are described by a quasipolynomial of degree $3a-2$ and period dividing $\ell$. In fact, we have checked that the period is exactly $l$ for $a$ less than 7.
\end{enumerate}
 \end{theorem}
 \begin{remark}
 Family 2 is included in the results concerning to $\mathcal{F}_a$, which is its generating functions after shifting the initial zeros.
 \end{remark}
 
 \begin{proof}
  \begin{enumerate}
 \item We define $P_a(x)$ as $P_a(x) = \mathcal{F}_a \cdot (1-x^l)^{2a}$.
 Then, the generating function $\mathcal{F}_a$ can be written as in the theorem.  
 Let  $\Phi_i$ be the $i^{th}$ cyclotomic polynomial. From the well--known identity $(x^n-1)= \prod_{i | n} \Phi_i$, we express $\mathcal{F}_a$ and $(1-x^l)^{2a}$ as product of cyclotomic polynomials. The cyclotomic polynomials appearing in $\mathcal{F}_a$ also appear in $(1-x^l)^{2a}$, with exponent at least equal to their exponent in $\mathcal{F}_a$. Then, $P_a$ is a polynomial and it can be written as a product of cyclotomic polynomials. 
 Moreover, $\deg (\mathcal{F}_a)= a(a+2)$, and then, $\deg (P_a)=2al - a(a+2)$. 
 
 \item We define $Q_a(x)$ as $Q_a(x) = \mathcal{G}_a \cdot (1-x^l)^{3a-1}$.
 Then, the generating function $\mathcal{G}_a$ can be written as in the theorem.
We express $\mathcal{G}_a$ and $(1-x^l)^{3a-1}$ as a product of cyclotomic polynomials, observing that the cyclotomic polynomials appearing in $\mathcal{G}_a$ also appear in $(1-x^l)^{3a-1}$, with at least equal exponent. Then, $Q_a$ is a polynomial and it can be written as a product of cyclotomic polynomials. 
 Moreover, $\deg (\mathcal{G}_a)= \frac{3}{2}a(a+1)$, and $\deg (Q_a)=l(a-1) - \frac{3}{2}a(a+1)$. 
 \end{enumerate}
The other two items follow using Proposition 4.13 of \cite{BecSan}.  
 \end{proof}
 
Let see some examples.
\begin{example}
The coefficients $\overline{g}^{(k)}_{(k^2), (k^2)}$ are given by the following quasipolynomial of degree $3$ and period $6$. 
\begin{align*}
\overline{g}^{(k)}_{(k^2), (k^2)} =
 \left\{
{\small{
 \begin{array}{ll}
 1/72 k^3 +1/6 k^2+\phantom{1}2/3  k +\phantom{5}1 \phantom{1}                            & \text{ if } k \equiv 0 \mod 6\\
 1/72 k^3 +1/6 k^2+13/24 k +5/18             & \text{ if } k \equiv 1 \mod 6\\
 1/72 k^3 +1/6 k^2+\phantom{1}2/3  k +8/9\phantom{1}              & \text{ if } k \equiv 2 \mod 6\\
 1/72 k^3 +1/6 k^2+ 13/24  k +1/2\phantom{1}                       &\text{ if } k \equiv 3 \mod 6\\
 1/72 k^3 +1/6 k^2+\phantom{1}2/3  k +7/9\phantom{1}                             &\text{ if } k \equiv 4 \mod 6\\
 1/72 k^3 +1/6 k^2+13/24  k +7/18                        &\text{ if } k \equiv 5 \mod 6
 \end{array}
  }}
 \right.
 \end{align*}
These quasipolynomials are computed applying the binomial identity to expand $(1-x^6)^4$, and then grouping the monomials in  $P_{2}=\Phi_2^2\Phi_3^3\Phi_6^4$ according to their degree mod $ 6$. For this, we write each number  as $n=2k+r$, with $r \in \{0,\ldots, 5\}$, and write the result in terms of the variable $k$.
 \end{example} 

 \smallskip
 
 \begin{example}
The coefficients $\overline{\overline{g}}_2(j)$ are given by the following quasipolynomial of degree $4$ and period $6$. 
\begin{align*}
\overline{\overline{g}}_2(j) =
 \left\{
{ \small{
 \begin{array}{ll}
 1/288 j^4 +1/16 j^3+\phantom{1}7/18  j^2 +\phantom{5} j + \phantom{1}1              & \text{ if } j \equiv 0 \mod 6\\
1/288 j^4 +1/16 j^3+\phantom{1}7/18  j^2 +15/16 j +175/288           & \text{ if } j \equiv 1 \mod 6\\
1/288 j^4 +1/16 j^3+\phantom{1}7/18  j^2 +\phantom{5}  j +8/9\phantom{1}              & \text{ if } j \equiv 2 \mod 6\\
1/288 j^4 +1/16 j^3+\phantom{1}7/18  j^2 + 15/16 j +23/32\phantom{1}                       &\text{ if } j \equiv 3 \mod 6\\
1/288 j^4 +1/16 j^3+\phantom{1}7/18  j^2 +\phantom{5} j +8/9\phantom{1}                             &\text{ if } j \equiv 4 \mod 6\\
1/288 j^4 +1/16 j^3+\phantom{1}7/18  j^2 +15/16  j +175/288                        &\text{ if } j \equiv 5 \mod 6
 \end{array}}}
 \right.
 \end{align*}
 \end{example}  

\begin{remark}
We denote by $m$ the power of $(1-x^l)$ in each case. It corresponds to the value of $2a$ for Family 1 and to the value of $3a-1$ for Family 3.
We compute the quasipolynomials following the proof of Proposition 4.13 in \cite{BecSan}. The algorithm is the following: apply the binomial identity to expand $(1-x^l)^m$, and then group the monomials in $P_a$, expressed as product of cyclotomic polynomials, according to their degree $\mod l$. For this, write each number as $n=a\cdot k+r$ or $n=a\cdot j+r$, depending on the variable of the quasipolynomials, with $r \in \{0,\dots, l-1\}$. Finally, write the result in terms of the variable of the quasipolynomials, $k$ for Family 1 or $j$ for Family 3. 

\end{remark}

\section{On the rate of growth of the coefficients} 

Thanks to the classical stability phenomenon for the Kronecker coefficients discovered by Murnaghan, we can transcribe our results in terms of reduced Kronecker coefficients as results about Kronecker coefficients.

Firstly we state them in the most interesting cases.
\begin{corollary}[Family 1, Kronecker coefficients version]\label{CorKC1}
Consider the Kronecker coefficients of the form
\begin{eqnarray*}
g_{(n-a\cdot k, k^a),(n-a\cdot k,k^a)}^{(n-k,k)} .
\end{eqnarray*}
Then, for $n\geq (a+3)\cdot k$, we have that
\begin{enumerate}
\item[(i)] Their generating function is
\begin{eqnarray*}
\mathcal{F}_a = \frac{1}{(1-x)(1-x^2)^2\cdots (1-x^a)^2(1-x^{a+1})}.
\end{eqnarray*}
\item[(ii)] They count the number of plane partitions of $k$ fitting inside a $2\times a$ rectangle. 
\item[(iii)] They can be described by a quasipolynomial of degree $2a-1$ and period dividing $l$ (least common multiple of $1,2, \dots ,a,a+1)$. 
\end{enumerate}
\end{corollary} 
For Family 3, we have also its corresponding result.
\begin{corollary}[Family 3, Kronecker coefficients version]\label{CorKC2}
Consider the Kronecker coefficients of the form
\begin{eqnarray*}
g_{(n-a\cdot k, k^a),\left(n-(a+1)\cdot k+j,2k-j,k^{a-1}\right)}^{(n-k,k)} .
\end{eqnarray*}
For $k\geq 2j$ and $n\geq (a+3)k$, we have that
\begin{enumerate}
\item[(i)] The generating function of these coefficients is
\begin{eqnarray*}
\mathcal{G}_a = \frac{1}{(1-x)^2(1-x^2)^3\cdots (1-x^{a-1})^3(1-x^a)^2(1-x^{a+1})}.
\end{eqnarray*}
\item[(ii)] They have the following combinatorial interpretation in terms of plane partitions: they are equal to
\begin{eqnarray*}
\sum_{l=0}^j \# \left\{ \begin{array}{c}  \text{plane partitions of } l \\ \text{in } 3\times (a-1) \text{ rectangle} \end{array} \right\} \# \left\{ \begin{array}{c}  \text{plane partitions of }j- l \\ \text{in } 2\times 1 \text{ rectangle}\end{array} \right\} .
\end{eqnarray*}
 
\item[(iii)] They can be described by a quasipolynomial of degree $3a-2$ and period dividing $l$ (least common multiple of $1,2, \dots ,a,a+1)$. 
\end{enumerate}
\end{corollary}
We finish this section with some observations about the rate of growth of the Kronecker coefficients. 

Murnaghan observed that the sequences obtained by adding cells to the first parts of the partitions indexing a Kronecker coefficients are eventually constant.
In \cite{BERARM}, they show that fixed three partitions, the Kronecker coefficients indexed by them stabilize when we increase these partitions with $n$ new boxes in their first row and $n$ new boxes in their first column. They also show that the resulting sequence obtained by increasing the sizes of the second rows (keeping the first one very long in comparison) of the partitions indexing the Kronecker coefficients are described by a linear quasipolynomial of period 2.

In \cite{Vergen15}, Vergne and Baldoni show that the \emph{stretching} reduced Kronecker coefficients $\overline{g}_{k\mu,k\nu}^{k\lambda}$, for any partitions $\lambda$, $\mu$ and $\nu$, are given by a quasipolynomial depending on $k$, $\lambda$, $\mu$ and $\nu$. 
For $i>0$, the reduced Kronecker coefficients coefficients $\overline{g}_{k\mu,k\nu}^{k\lambda +(i)}$, for any partitions $\lambda$, $\mu$ and $\nu$, are given by piecewise quasipolynomials depending on $k$, $\lambda$, $\mu$ and $\nu$, \cite{Manivel14, zbMATH01307854}.
In \cite{EACA14}, the family of reduced Kronecker coefficients $\left\{\overline{g}_{k\lambda,k\mu}^{k\nu+(i)}\right\}_{k,i\geq 0}$ is described in terms of linear piecewise quasilinear polynomials of period 2, including explicit formulas when the partitions $\lambda$, $\mu$ and $\nu$ have length at most 1. 

  An interesting question is then to describe what happens when we add cells to arbitrary rows of the partitions indexing a Kronecker (and reduced Kronecker) coefficient. The results presented in this paper show several cases when we know what happen. For example, for $a=1$, the three families of sequences are describes by a linear quasipolynomial of period 2, as is predicted in the work of Briand, Rattan and Rosas, \cite{BERARM}. But when $a=2$, the sequence corresponding to Family 1 is described by a quasipolynomial of degree 3 and the one corresponding to Family 3 is described by a quasipolynomial of degree 4. For $a=3$, the sequences are described by quasipolynomials of degree 5 and degree 7 (respectively), and so on. 
  
  As a final remark, when we increase the parameter $a$ in case of $b=a$ in Family 1 and $b=a+1$ in Family 3, the sequences are weakly increasing. This corresponds to increase the size of the columns in the indexing partitions.
\bibliographystyle{alpha}
\bibliography{Referencias_Articulos}

\end{document}